\documentclass[12pt]{amsart}

\usepackage{amsmath,amsfonts}
\newtheorem{thm}{Theorem}[section]
\newtheorem{lemma}[thm]{Lemma}
\newtheorem{prop}[thm]{Proposition}

\newtheorem{defn}[thm]{Definition}
\newcommand{\R}{\mathbb{R}}
\newcommand{\A}{\alpha}
\newcommand{\E}{\epsilon}
\newcommand{\Z}{\begin{equation}}
\newcommand{\z}{\end{equation}}
\newcommand{\y}{\overline}
\newcommand{\IP}{\mathbb{P}}
\newcommand{\IE}{\mathbb{E}}

\begin{document}

\title[Differential equations driven by rough paths]{Differential equations
driven by rough paths: an approach via discrete approximation.}

\author{A. M. Davie}

\address{School of Mathematics, University of Edinburgh, King's Buildings,
Mayfield Road, Edinburgh EH9 3JZ, UK}

\subjclass[2000]{Primary 60H10; Secondary 34F05}



\begin{abstract}A theory of systems of differential equations of the
form $dy^i=\sum_jf^i_j(y)dx^i$, where the driving path $x(t)$ is 
non-differentiable, has recently been developed by Lyons. We develop an
alternative approach to this theory, using (modified) Euler approximations,
and investigate its applicability to stochastic differential equations
driven by Brownian motion. We also give some other examples showing that
the main results are reasonably sharp.
\end{abstract}

\maketitle

\section{Introduction}\label{int}

Lyons \cite{tjl} has developed a theory of systems of differential equations
of the form
\Z\label{eq1} dy^i=\sum_{j=1}^df_j^i(y)dx^j,\ \ \ \ \ \ \ y^i(0)=y_0^i,
\ \ \ i=1,2,\cdots,n\z
where $x(t)$ is a given continuous (vector-valued) function of $t$ but
is not assumed to be differentiable, so the system is not a system of
classical differential equations. In \cite{tjl} $x(t)$ is assumed to have
finite $p$-variation for some positive $p$. The study of equations driven by such
rough paths is motivated by the case of stochastic differential equations driven by
Brownian motion, which has finite $p$-variation only if $p<\frac12$. Rough path theory
gives an approach to such stochastic equations by viewing them as deterministic equations,
for a fixed choice of driving path, which is in contrast to the more classical stochastic
approach. Applications of rough path theory to stochastic equations can be found, for example,
in \cite{fr,hl,lq}.

The approach to (\ref{eq1}) in \cite{tjl}
mirrors the standard approach to ODE's by writing them as integral equations
and using Picard iteration or a contraction mapping argument. So one writes
\[y^i(t)=y^i_0+\sum\int_0^tf^i_j(y(s))dx^j(s)\]
and the problem then is to interpret the integral on the right. This is
fairly straightforward if $p<2$; if $2\leq p<3$ then Lyons shows that one
can make sense of this integral if one `knows' integrals of the form
$\int x^i(s)dx^j(s)$. His approach is to suppose these latter integrals
are given, subject to natural consistency conditions, and then to develop
an integration theory which suffices to treat the differential
equations. If $p\geq3$ then it is necessary to assume higher-order
iterated integrals of $x(t)$ are given. In this setting Lyons in \cite{tjl}
proves existence and uniqueness of solutions of (\ref{eq1}) provided $f\in
C^\gamma$ where $\gamma>p$.

Another way of proving existence and uniqueness theorems for classical
ODE's of the form $\dot{y}=f(y)$ is to consider Euler approximations
$y_{k+1}=y_k+(t_{k+1}-t_k)f(y_k)$ associated with subdivisions
$\{t_0,t_1,\cdots t_K\}$ and show that as the subdivision gets finer
these approximations converge to a limit which satisfies the equation.
Indeed this is a standard proof for the case where $f$ is continuous
but not Lipschitz, so that one gets existence but not uniqueness. In the
present paper we use this approach to study the system (\ref{eq1}). Using
suitable estimates for discrete approximations to (\ref{eq1}) we are able to
prove their convergence to solutions of (\ref{eq1}). When $p<2$ the simple
Euler approximation suffices. When $2\leq p<3$ we need to assume the integrals
$\int x^idx^j$ given, and incorporate them into the discrete approximation.
We restrict attention to $p<3$ which is enough to illustrate the basic
ideas and covers many of the applications, but avoids the algebra of
iterated integrals which is needed for the general case. We show (in
sections \ref{s2} and \ref{s3}) existence of solutions when $f\in
C^{\gamma-1}$ when $\gamma>p$, and uniqueness when $f\in C^\gamma$ where
$\gamma\geq p$. The proofs give, when $f\in C^\gamma$, convergence of the
Euler approximation to the solution in the case $p<2$, and convergence of the
modified Euler approximation when $2\leq p<3$.

We treat the simpler case $1\leq p<2$ first, in section \ref{s2}, the
treatment of $2\leq p<3$ in section \ref{s3} being similar but with extra
terms to handle. This results in some repetition of arguments but it is
hoped that treating the simpler case will make the ideas clearer.

In section \ref{bm} we consider the application of the theory to the case of
equations driven by Brownian motion, which is one of the motivating
examples. In this context we investigate the smoothness requirements on
$f$ for existence and uniqueness of solutions.
In section \ref{oe} we give examples to show that the results of sections
\ref{s2} and \ref{s3} are sharp, to the extent that that uniqueness can fail
with $f\in C^\gamma$ whenever $1<\gamma<p<3$, and existence can fail for
$f\in C^{p-1}$ whenever $1<p<2$ or $2<p<3$.  Section \ref{ge} treats global
existence questions. Finally in Section \ref{ce} we show that, under an
additional condition, the (unmodified) Euler approximations, with uniform
step size, converge to the solution even when $2\leq p<3$.

We comment briefly on the relation of our results to those of \cite{tjl}. When
$p\geq2$ the notion of solution developed in \cite{tjl} contains more than the solution
path $y(t)$ which we obtain. Just as the driving path $x(t)$ has to be accompanied by the associated
`iterated integrals', the approach taken in \cite{tjl} is that the solution $y(t)$ should also be
accompanied by its iterated integrals, and the solution obtained there incorporates these as well
as $y(t)$. This leads to a more complete theory at the expense of greater complexity. The results
of sections \ref{s2} and \ref{s3} recover the main existence and uniqueness result of \cite{tjl} for $\gamma<3$, in the
restricted sense that only the solution path $y(t)$ is obtained, by a different method and somewhat
more easily. We get some slight improvement in the regularity requirements, in that existence is shown
for $f\in C^{\gamma-1}$ rather than $C^\gamma$ and uniqueness in the borderline case $\gamma=p$ (for $f\in C^\gamma$).
Some more discussion of the relation to \cite{tjl}, including an indication of how the full solution of \cite{tjl} might be
obtained by our method, can be found at the end of section \ref{s3}. The results of the remaining sections are new to the best of
my knowledge.
\vspace{.2cm}\\
{\bf Notation.}

Let $n$ and $d$ be positive integers.
For $\rho=N+\alpha$ where $N$ is a positive integer and $0<\alpha\leq1$ we
define $C^\rho$ to be the set of $f=(f_j^i)$ where $i=1,\cdots,n$ and
$j=1,\cdots,d$, such that each $f_j^i$ is defined on $\R^n$,
 and has derivatives up to order $N$, which all satisfy locally a H\"{o}lder
condition of exponent $\alpha$ (note that this definition of $C^N$ when
$N$ is an integer differs from the usual, in that we do not require
continuity of the derivatives of order $N$). We denote by $C_0^\rho$ the
set of $f\in C^\rho$ which vanish outside a bounded set.

We shall generally suppose that either $1\leq p<\gamma\leq2$ or $2\leq
p<\gamma\leq3$.

Consider a continuous $\R^d$-valued function $x(t)=(x^1(t),\cdots,x^d(t))$ on an
interval $[0,T]$, and suppose $x(t)$ has finite $p$-variation, in the
sense that there is a continuous increasing function $\omega$ on $[0,T]$
such that $|x(t)-x(s)|^p\leq\omega(t)-\omega(s)$ whenever $0\leq s\leq
t\leq T$. We sometimes write $\omega(s,t)$ for $\omega(t)-\omega(s)$.

We use the summation convention for indices $h,q,r,j$, where $h$ and $q$
range from 1 to $n$ and $r$ and $j$ from 1 to $d$. Then (\ref{eq1}) can be written
\[dy^i=f_j^i(y)dx^j,\ \ \ \ \ \ \ y^i(0)=y_0^i\]
where $y(t)=(y^1(t),\cdots,y^n(t))$. We also write $\partial_q$ for
$\partial/\partial y_q$.

We shall occasionally use {\em dyadic intervals}, by which we mean
intervals of the form\\
$[k2^{-m},(k+1)2^{-m}]$ where $k$ and $m$ are integers.\vspace{.2cm}\\
\section{\bf The case $p<2$.}\label{s2}

Suppose $1\leq p<\gamma\leq2$, and that $x(t)$ has finite $p$-variation
in the sense defined above. Then we interpret (1) as follows:

\begin{defn}\label{def1}
We say $y(t)$ is a {\em solution} of (1) on $[0,T]$ if $y^i(0)=y_0^i$ and there exists a
continuous increasing function $\tilde{\omega}$ on $[0,T]$ and a non-negative
function $\theta$ on $[0,\infty)$ such that $\theta(\delta)=o(\delta)$ as
$\delta\rightarrow 0$ and such that
\Z\label{eq2}|y^i(t)-y^i(s)-f_j^i(y(s))(x^j(t)-x^j(s))|\leq\theta
(\tilde{\omega}(t)-\tilde{\omega}(s))\z
for all $s,t$ with $0\leq s<t\leq T$.
\end{defn}

Note that $\tilde{\omega}$ may {\em a priori} differ from the function
$\omega$ that appears in the $p$-variation condition on $x$, though we
shall in fact see (remark 1 below) that, for any solution $y$, (2)
will actually hold with $\tilde{\omega}=\omega$.

We also consider discrete approximations to a solution: given
$0=t_0<t_1<\cdots<t_K$, let $x_k=x(t_k)$ and, given $y_0$, define $y_k$
by the recurrence relation
\Z\label{eq3}y_{k+1}^i=y_k^i+f_j^i(y_k)(x_{k+1}^j-x_k^j)\z
Then we can state the following existence result:

\begin{thm}\label{th2}
 Let $f\in C^{\gamma-1}$ (assuming $1\leq p<\gamma\leq2$) and
$y_0\in\R^n$. Then there
exists $\tau$, with $0<\tau\leq T$, and a solution $y(t)$ of (\ref{eq1})
for $0\leq t<\tau$, such that if $\tau<T$ then $|y(t)|\rightarrow\infty$
as $t\rightarrow\tau$.
\end{thm}

Example 3 in section \ref{oe} shows that Theorem \ref{th2} is sharp to the
extent that existence can fail for $f\in C^{p-1}$, $1<p<2$. (However when $p=1$, i.e. $x$ has bounded variation, it can be improved;
in this case it suffices that $f$ be continuous).\vspace{.2cm}

\begin{thm}\label{th1} Let $f\in C^\gamma$ and $y_0\in\R^n$. Then the
solution $y(t)$ of (\ref{eq1}) given by Theorem \ref{th2} is unique in the
sense that if $t<\tau$ and $\tilde{y}$ is another solution of (\ref{eq1})
on $[0,t]$ then $\tilde{y}=y$ on $[0,t]$.

Moreover, if $t<\tau$ and $\epsilon>0$, we can find $\delta>0$ so that
if $0=t_0<\cdots<t_K=t$ with $t_k-t_{k-1}<\delta$ for each $k$, then
\[|y_k-y(t_k)|<\epsilon\]
for each $k$, where $y_k$ is given by (\ref{eq3}).
\end{thm}

Example 1 in section \ref{oe} shows that, if $\gamma<p$, uniqueness fails for a
suitable choice of $f\in C^\gamma$.

These theorems will be proved by analysis of the discrete approximation
(\ref{eq3}).\vspace{.3cm}\\
{\bf Analysis of the discrete problem (\ref{eq3}).}

Let $x_0,\cdots,x_K\in\R^d$, with $x_k=(x_k^1,\cdots,x_k^d)$, and
suppose for $0\leq k\leq l\leq K$ we are given $\omega_{kl}\geq 0$ such that
$w_{km}\geq\omega_{kl}+\omega_{lm}$ when $k\leq l\leq m$ and
$|x_l^r-x_k^r|^p\leq\omega_{kl}$.

Note that with $\omega_{kl}=\omega(t_l)-\omega(t_k)$ and $x_k=x(t_k)$ this is
just what we get from the problem described in the previous section. However
for the purposes of this section we can forget about the variable $t$.

Given $y_0\in\R^n$, and $f\in C_0^{\gamma-1}$ we define $y_k\in\R^n$ for
$k=1,2,\cdots,K$  by the recurrence relation (\ref{eq3}). Observe that, if we
let $z_k^{iN}=\partial y_k^i/\partial y_0^N$ then
\[z_{k+1}^{iN}=z_k^{iN}+\partial_qf_j^i(y_k)z_k^{qN}(x_{k+1}^j-x_k^j)\]
which is the recurrence relation of the type (\ref{eq3}) obtained when $f$ is
replaced by $F(y)=(\{f_j^i(y)\},\{\partial_qf_j^i(y)\})$, taking values
in $\R^{n(n+1)}\times\R^d$. This observation enables us to apply results
on solutions of (\ref{eq3}) to the derivatives $z_k^{im}$.

Let
\[I_{kl}^i=y_l^i-y_k^i-f_j^i(y_k)(x_l^j-x_k^j)\]

The following is our main technical result:

\begin{lemma}\label{l1} (a) There are positive numbers $C$ and $M$,
depending only on $n,d,\gamma,p,\omega_{0K}$ and $\|f\|_{\gamma-1}$, such that if
$0\leq k\leq l\leq K$ then $|I_{kl}^i|\leq M\omega_{kl}^{\gamma/p}$
and $|y_l-y_k|\leq C\omega_{kl}^{1/p}$.\vspace{.2cm}\\
(b) Suppose now $f\in C_0^\gamma$. Let $\tilde{y}\in\R^n$ and let
$\tilde{y}_k$ be the corresponding solution of the recurrence relation.
There is $M'>0$, depending only on $n,d,\gamma,p,\omega_{0K}$ and
$\|f\|_\gamma$, such that if $0\leq k\leq K$ then
\[|\tilde{y}_k^i-y_k^i|\leq M'\max_i|\tilde{y}_0^i-y_0^i|\]
\end{lemma}

\begin{proof}
We will use $B_1$, $B_2$ etc to denote constants depending
only on $n,d,\gamma,p$ and the $C^{\gamma-1}$-norm of $f$.

If $k\leq l\leq m$ we have
\Z\label{eq4}I_{km}^i-I_{kl}^i-I_{lm}^i=(f_j^i(y_l)-f_j^i(y_k))(x_m^j-x_l^j)\z
{\bf Claim.} If $\delta>0$ is small enough and $L$ large enough
(depending only on $n,d,\gamma,p$ and the norm of $f$) then $|I_{km}^i|\leq
L\omega_{km}^{\gamma/p}$ whenever $\omega_{km}\leq\delta$.
\vspace{.2cm}\\
{\bf Proof of claim.} We use induction on $m-k$. Note first that
$I_{k,k}=0$ trivially and $I_{k,k+1}=0$ by (\ref{eq3}). Now suppose $k,m$
chosen with $m-k>1$ and suppose the claim holds for smaller values of $m-k$.

 Let $l$ be the largest integer with $k\leq
l<m$ satisfying $\omega_{kl}\leq\frac 12\omega_{km}$. Then
$\omega_{k,l+1}>\frac 12\omega_{km}$ so $\omega_{l+1,m}<\frac
12\omega_{km}$. Then by the inductive hypothesis the claim holds for
$k,l$, i.e. $I_{kl}\leq L\omega_{kl}^{\gamma/p}$. Then $|y_l-y_k|\leq
L\omega_{kl}^{\gamma/p}+B_1\omega_{kl}^{1/p}$. Then provided
\Z\label{eq5}L\delta^{(\gamma-1)/p}\leq B_1\z we have
\Z\label{eq6}|y_l-y_k|\leq2B_1\omega_{kl}^{1/p}\z 
Putting these estimates into (\ref{eq4}) we find that
\[|I_{km}^i|\leq|I_{kl}^i|+|I_{lm}^i|+B_2\omega_{km}^{\gamma/p}\]
In the same way we find that
$|I_{lm}^i|\leq|I_{l,l+1}^i|+|I_{l+1,m}^i|+B_2\omega_{km}^{\gamma/p}$. 
Since $I_{l,l+1}=0$ we get
\[|I_{km}^i|\leq|I_{kl}^i|+|I_{l+1,m}^i|+2B_2\omega_{km}^{\gamma/p}\leq
L(\omega_{kl}^{\gamma/p}+\omega_{l+1,m}^{\gamma/p})+2B_2\omega_{km}
^{\gamma/p}\leq(2^{1-\gamma/p}L+2B_2)\omega_{km}^{\gamma/p}\]
and provided \Z\label{eq7}(1-2^{1-\gamma/p})L\geq 2B_2\z we conclude that
$|I_{km}^i|\leq L\omega_{km}^{\gamma/p}$ which completes the induction,
provided we choose $L$ and $\delta$ to satisfy (\ref{eq7}) and (\ref{eq5}).
The claim is proved.
 
For intervals with $\omega_{kl}\leq\delta$, part (a) of the lemma now
follows from the claim, and the fact that the proof of (\ref{eq6}) now holds
for any $k,l$. If $\omega_{kl}>\delta$, we can decompose
$k=k_0<k_1<\cdots<k_r=l$ where either $\omega_{k_uk_{u+1}}\leq\delta$ or
$k_{u+1}=k_u+1$ for each $u$, and $r\leq1+2\delta^{-1}\omega_{kl}$. In
either case $|y_{k_{u+1}}-y_{k_u}|\leq2B_1\omega_{k_uk_{u+1}}^{1/p}$.
Summing gives $|y_l-y_k|\leq(1+2\delta^{-1}\omega_{kl})2B_1\omega_{kl}^{1/p}$,
and then $|I_{kl}|\leq|y_l-y_k|+B_1\omega_{kl}^{1/p}\leq{\rm
const}\ \omega_{kl}^{\gamma/p}$, using the fact that $\omega_{kl}>\delta$.

To prove (b), we suppose $f\in C_0^\gamma$, and apply (a), using the
observation above, to estimate $z_k^{im}$. We find that, for any choice of
$y_0$, we have $|z_k^{im}|\leq{\rm const}\ \omega_{0k}^{1/p}\leq1$ and (b)
follows.
\end{proof}
{\em Proof of theorems.} We prove theorem \ref{th2} first. Suppose $f\in
C^{\gamma-1}$. Then for $r=1,2,\cdots$ we can find $f_{(r)}\in C^{\gamma-1}
_0$ with $f_{(r)}(y)=f(y)$ for $|y|\leq r$. Now take a sequence of
successively finer partitions $\{{\mathcal P}_m:\ m=1,2,\cdots\}$ of $[0,T]$
with mesh tending to 0. Let $y^{(m)}_k$, which we also write as $y^{(m)}
(t_k)$ be the solution of (3) using the partition ${\mathcal P}_m$. By
passing to a subsequence we can assume that $y^{(m)}(s)$ converges to a limit
$y(s)$ (possibly $\pm\infty$) for each $s\in\cup_m{\mathcal P}_m$.
Let $\tau_r=\sup\{t:\ 0\leq t<T,$ there exists $m_0$ such that $|y^{(m)}
(s)|<r$ for all $m>m_0$ and all $s\in{\mathcal P}_m$ with $0\leq s\leq t\}$;
by applying Lemma \ref{l1}(a) to $f_{(r)}$ we see that $\tau_r$ is
well-defined and positive for all $r>|y_0|$. Also from this lemma, if $0\leq
t<\tau_r$ we have $|y^{(m)}(s)-y^{(m)}(s')|\leq C(r,t)|\omega(s)-\omega(s')|
^{1/p}$ for $s,s'\in{\mathcal P}_m$ with $0\leq s,s'\leq t$ and $m$ large
enough. We then have the same bound for $|y(s)-y(s')|$, and $y(s)$ extends
to $[0,t]$ by continuity. By Lemma \ref{l1}(a) again the bound (\ref{eq2})
holds on $[0,t]$, for every $t<\tau_r$, and for every $r$. Now let
$\tau=\lim\tau_r$; it follows that $y$ is a solution of (1) on $[0,\tau)$.

Now suppose $\tau<T$. It follows from Lemma \ref{l1}(a) applied to $f_{(r+1)}$
that for each $r$ there is $\sigma_r>0$ such that if $s\in{\mathcal P}_m$ for
some $m$ and if $|y^{(m)}|<r$ then $|y^{(m)}(s)|\leq r+1$ for any $s\in
{\mathcal P}_m$ satisfying $t<s<t+\sigma_r$. Next, fix $r$ and choose $t\in
\cup_m{\mathcal P}_m$ such that $t>\tau-\sigma_r$. Then if $|y(t)|<r$ there
exists $m_0$ such that $|y^{(m)}(t)|<r$ for $m>m_0$, and then
$|y^{(m)}(s)|\leq r+1$ for all $m>m_0$ and $s\in{\mathcal P}_m$ with
$t<s<t+\sigma$, contradicting the definition of $\tau$.

So $|y(t)|\geq r$ for $\tau-\sigma_r<t<\tau$, so $|y(t)|\rightarrow\infty$ as
$t\rightarrow\tau$. This proves Theorem \ref{th2}.\vspace{.2cm}

To prove Theorem \ref{th1}, suppose $y$ is a solution of (\ref{eq1}) on
$[0,t]$ and consider a partition $0=t_0<\cdots<t_K=t$. Choose $r$ so that $|y(\tau)|<r$ for
$\tau\in[0,t]$. For $l\geq k$ let $z_l^{(k)}$ be the solution of (\ref{eq3}), with $f_{(r)}$
in place of $f$, with initial value $z_k^{(k)}=y(t_k)$. Then by (\ref{eq2}) $|z_{k+1}^{(k)}-y_{k+1}|
\leq\theta(\tilde{\omega}(t_{k+1})-\tilde{\omega}(t_k)$ and using Lemma \ref{l1}(b) we have for $k<l\leq K$ that
$|z_l^{(k)}-z_l^{(k+1)}|\leq{\rm const}\ \theta(\tilde{\omega}
(t_{k+1})-\tilde{\omega}(t_k))$.
Summing over $k$ we deduce a bound for $z_l^{(0)}-y_l$ which tends to 0 as the
mesh of the partition tends to 0. The conclusions of Theorem \ref{th1} follow.
\vspace{.2cm}\\
{\bf Remark 1.} The estimates given by Lemma \ref{l1} show that any solution
constructed by the method described in the above proof will satisfy the
following stronger form of (\ref{eq2}):
\Z\label{eq8}|y^i(t)-y^i(s)-f_j^i(y(s))(x^j(t)-x^j(s))|=O(\omega(t)-\omega(s))
^{\gamma/p}\z
If $f\in C^\gamma$ then the uniqueness shows that any solution will
satisfy this stronger inequality. Then we can take
$\theta(\delta)=M\delta^{\gamma/p}$ for a suitable constant $M$, and the
proof of Theorem 1 then gives the bound
\Z\label{eq9}|y_k-y(t_k)|\leq C\sum_{j=1}^k\omega_{j-1,j}^{\gamma/p}\z

When we only have $f\in C^{\gamma-1}$,
then we can still show that any solution satisfies (\ref{eq8}), by
using a modified form of Lemma \ref{l1}, as follows:

Given a solution $y$ of (\ref{eq1}) satisfying (\ref{eq2}) on $[0,T]$, and
given $\epsilon>0$, choose $\eta>0$ so that $\theta(\eta)<\epsilon\eta$,
and choose a partition with $\tilde{\omega}_{k,k+1}<\delta$. Then writing
$y_k=y(t_k)$, we have
\[|y_{k+1}^i-y_k^i-f_j^i(y_k)(x_{k+1}^j-x_k^j)|\leq\epsilon\tilde{\omega}
_{k,k+1}\]
and using this instead of (\ref{eq3}) we follow the proof of Lemma \ref{l1};
we obtain $|I_{km}^i|\leq L\omega_{km}^{\gamma/p}+\epsilon\tilde{\omega}_{km}$
whenever $\omega_{km}\leq\delta$ and $\tilde{\omega}_{km}\leq\epsilon^{-1/2}
\delta$, where now $\delta$ should satisfy $\epsilon^{1/2}\delta^{1-1/p}+
L\delta^{(\gamma-1)/p}\leq B_1$ and $L$ satisfies (\ref{eq7}) as before.
Letting $\epsilon\rightarrow0$ gives the same estimates as the original form
of Lemma \ref{l1}.\vspace{.2cm}\\
{\bf Remark 2.} If $f\in C^\gamma$, the Euler approximations $y^{(m)}$
used in the proof of Theorem 1 converge to the solution, without the
need to pass to a subsequence. This follows from (\ref{eq9}),
or proved directly as follows.

As in the proof of Theorem \ref{th1}, let $y^{(m)}_k=y^{(m)}(t_k)$ be the
solution of (\ref{eq3}) corresponding to the partition ${\mathcal P}_m$ given
by $0=t_0<t_1<\cdots<t_K$. Let $m'>m$ so that ${\mathcal P}_{m'}$ is a finer
partition. Let $v_k$ be the solution of (\ref{eq3}) for the partition
${\mathcal P}_{m'}$, at the point $t_k$ (which being a point of ${\mathcal 
P}_m$ is also a point of ${\mathcal P}_{m'}$. Then for $l\geq k$ let
$z_l^{(k)}$ be the solution of (\ref{eq3}), for the partition ${\mathcal 
P}_m$, with initial condition $z_k^{(k)}=v_k$. Then by the bound for
$I^i_{kl}$ in Lemma 1(a), applied to the partition ${\mathcal P}_{m'}$, we
have $|z_{k+1}^{(k)}-v_{k+1}|\leq c\omega_{k,k+1}^{\gamma/p}$ and the result
follows by using Lemma 1(b) and summing over $k$, as in the last part of the
proof of Theorem \ref{th1}.

The above argument may be useful for the generalisation of Theorem \ref{th2}
to an infinite-dimensional setting, where the compactness required for the
proof of Theorem \ref{th1} may fail.

\section{The case $2\leq p<3$.}\label{s3}

We now suppose $2\leq p<\gamma\leq3$. In this case (\ref{eq3}) does not give
a sufficiently good approximation, and we need to include higher-order
terms. We can regard (\ref{eq3}) as being obtained from (\ref{eq1}) by
approximating $f_j^i(y)$ by $f_j^i(y_k)$. A better approximation is
\[f_j^i(y)\approx f_j^i(y_k+f_r(y_k)(x^r-x_k^r))\approx f_j^i(y_k)+
\partial_hf_j^i(y_k)f_r^h(y_k)(x^r-x_k^r)\]
To solve (\ref{eq1}) using this approximation, we have to integrate
$(x^r-x_k^r)dx^j$. With this as motivation, we attempt to define
$A^{rj}(s,t)$ for $s\leq t$ by $dA^{rj}(s,t)=\{x^r(t)-x^r(s)\}dx^j(t)$
with $A^{r,j}(s,s)=0$. We have then the problem of interpreting this
equation. The solution adopted in \cite{tjl} is to make the following
assumption:\vspace{.2cm}\\
{\bf Assumption 1.} We suppose as given the
quantities $A^{rj}(s,t)$ for $1\leq r,j\leq d$ and $0\leq s\leq t\leq T$
subject to the natural consistency condition
\[A^{rj}(s,u)=A^{rj}(s,t)+A^{rj}(t,u)+(x^r(t)-x^r(s))(x^j(u)-x^j(t))\]
whenever $s\leq t\leq u$. We also assume the bound
$|A^{rj}(s,t)|^{p/2}\leq\omega(t)-\omega(s)$ (redefining $\omega(t)$ if
necessary).\vspace{.2cm}

We remark that, at least if $p>2$, it is not hard to prove the existence
of such $A^{rj}(s,t)$ satisfying the above conditions, for a given
choice of $x(t)$. There will be many possible choices of $A^{rj}(s,t)$;
given one such, then $\tilde{A}^{rj}(s,t)=A^{rj}(s,t)+\rho(t)-\rho(s)$
will be another, as long as $\rho(t)$ has finite $\frac p2$-variation.
Different choices lead to different interpretations of (\ref{eq1}).

We now interpret (\ref{eq1}) as follows:
\begin{defn}\label{def2}
We say $y(t)$ is a solution of (1) on $[0,T]$ if $y^i(0)=y_0^i$ and there exists a
continuous increasing function $\tilde{\omega}$ on $[0,T]$ and a non-negative
function $\theta$ on $[0,\infty)$ such that $\theta(\delta)=o(\delta)$ as
$\delta\rightarrow 0$ and such that
\Z\label{eq10}|y^i(t)-y^i(s)-f_j^i(y(s))(x^j(t)-x^j(s))-f_r^h(y(s))\partial_h
f_j^i(y(s))A^{rj}(s,t)|\leq\theta(\tilde{\omega}(t)-\tilde{\omega}(s))\z
for all $s,t$ with $0\leq s<t\leq T$.
\end{defn}

As before consider discrete approximations to a solution: given
$0=t_0<t_1<\cdots<t_K$, let $x_k=x(t_k)$ and, given $y_0$, define $y_k$
by the recurrence relation
\Z\label{eq11}y_{k+1}^i=y_k^i+f_j^i(y_k)(x_{k+1}^j-x_k^j)+f_r^h(y_k)\partial_h
f_j^i(y_k)A^{rj}(t_k,t_{k+1})\z
Then we can state the following:

\begin{thm}\label{th4}Let $f\in C^{\gamma-1}$ and $y_0\in\R^n$. Then there
exists $\tau$, with $0<\tau\leq T$, and a solution $y(t)$ of (1)
for $0\leq t<\tau$, such that if $\tau<T$ then $|y(t)|\rightarrow\infty$
as $t\rightarrow\tau$.
\end{thm}

\begin{thm}\label{th3} Let $f\in C^\gamma$ and $y_0\in\R^n$. Then the
solution $y(t)$ of (\ref{eq1}) given by Theorem \ref{th4} is unique in the
sense that if $t<\tau$ and $\tilde{y}$ is another solution of (\ref{eq1})
on $[0,t]$ then $\tilde{y}=y$ on $[0,t]$.
for $0\leq t<\tau$, such that either $\tau=T$ or $|y(t)|\rightarrow\infty$
as $t\rightarrow\tau$.

Moreover, if $t<\tau$ and $\epsilon>0$, we can find $\delta>0$ so that
if $0=y_0<\cdots<t_K=t$ with $t_k-t_{k-1}<\delta$ for each $k$, then
\[|y_k-y(t_k)|<\epsilon\]
for each $k$, where $y_k$ is given by (\ref{eq11}).
\end{thm}

These results will be proved by analysis of the discrete approximation
(\ref{eq11}).\vspace{.3cm}\\
{\bf Analysis of the discrete problem (\ref{eq11}).}

Let $x_0,\cdots,x_K\in\R^d$, with $x_k=(x_k^1,\cdots,x_k^d)$, and
suppose for $0\leq k\leq l\leq K$ given $A_{kl}^{rj}$ for $1\leq r,j\leq
d$ such that
$A_{km}^{rj}=A_{kl}^{rj}+A_{lm}^{rj}+(x_l^r-x_k^r)(x_m^j-x_l^j)$
whenever $k\leq l\leq m$. Suppose also given $\omega_{kl}\geq 0$ for
$0\leq k\leq l\leq K$ such that $w_{km}\geq\omega_{kl}+\omega_{lm}$ when
$k\leq l\leq m$ and $|x_l^r-x_k^r|^p\leq\omega_{kl}$ and
$|A^{rj}_{kl}|\leq\omega_{kl}^{2/p}$.

Given $y_0\in\R^n$, and $f\in C_0^{\gamma-1}$ we define $y_k\in\R^n$ for
$k=1,2,\cdots,K$  by the recurrence relation (\ref{eq11}). Observe that, if we
let $z_k^{iN}=\partial y_k^i/\partial y_0^N$ as before, then
\[z_{k+1}^{iN}=z_k^{iN}+\partial_qf_j^i(y_k)z_k^{qN}(x_{k+1}^j-x_k^j)
+\{\partial_qf_r^h(y_k)z_k^{qN}\partial_hf_j^i(y_k)+f_r^h(y_k)
\partial_{qh}f_j^i(y_k)z_k^{qN}\}A^{rj}\]
which is the recurrence relation of the type (\ref{eq11}) obtained when $f$ is
replaced by $F(y)=(\{f_j^i(y)\},\{\partial_qf_j^i(y)\})$, taking values
in $\R^{n(n+1)}\times\R^d$. 

Let
\[J_{kl}^i=y_l^i-y_k^i-f_j^i(y_k)(x_l^j-x_k^j)-f_r^h(y_k)\partial_h
f_j^i(y_k)A_{k,l}^{rj}\]

The following is our main technical result:

\begin{lemma}\label{l2} (a) There are positive numbers $C$ and $M$,
depending only on $n,d,\gamma,p,\omega_{0K}$ and $\|f\|_{\gamma-1}$, such
that if $0\leq k\leq l\leq K$ then $|J_{kl}^i|\leq
M\omega_{kl}^{\gamma/p}$ and $|y_l-y_k|\leq C\omega_{kl}^{1/p}$.\vspace{.2cm}\\
(b) Suppose now $f\in C_0^\gamma$. Let $\tilde{y}\in\R^n$ and let
$\tilde{y}_k$ be the corresponding solution of the recurrence relation.
There is $M'>0$, depending only on $n,d,\gamma,p,\omega_{0K}$ and
$\|f\|_\gamma$, such that if $0\leq k\leq K$ then \[|\tilde{y}_k^i-y_k^i|\leq
M'\max_i|\tilde{y}_0^i-y_0^i|\]
\end{lemma}
\begin{proof} We will use $B_1$, $B_2$ etc to denote constants depending
only on $n,d,\gamma,p$ and the $C^{\gamma-1}$-norm of $f$. 
Let $I_{kl}^i$ be as before and let
$g_{rj}^i(y)=f_r^h(y)\partial_hf_j^i(y)$ so that
\[J_{kl}^i=I_{kl}^i-g_{rj}^i(y_k)A_{kl}^{rj}\]
Also define
\[R_{kl,j}^i=f_j^i(y_l)-f_j^i(y_k)-\partial_hf_j^i(y_k)(y_l^h-y_k^h)\]
and note that $|R_{kl,j}^i|\leq B_1|y_l-y_k|^{\gamma-1}$.  Then if $k\leq
l\leq m$ we have
\Z\label{eq12}J_{km}^i-J_{kl}^i-J_{lm}^i=(R_{kl,j}^i+\partial_hf_j^i(y_k)
I_{kl}^h)(x_m^j-x_l^j)+\{g_{rj}^i(y_l)-g_{rj}^i(y_k)\}A_{lm}^{rj}\z
{\bf Claim.} If $\delta>0$ is small enough and $L$ large enough
(depending only on $n,d,\gamma,p$ and the norm of $f$) then $|J_{km}^i|\leq
L\omega_{km}^{\gamma/p}$ whenever $\omega_{km}\leq\delta$.
\vspace{.2cm}\\
{\bf Proof of claim.} Again we use induction on $m-k$. Note first that
$J_{k,k}=0$ trivially and $J_{k,k+1}=0$ by (\ref{eq11}). Now suppose $k,m$ chosen
with $m-k>1$ and suppose the claim holds for smaller values of $m-k$.

Let $l$ be the largest integer with $k\leq
l<m$ satisfying $\omega_{kl}\leq\frac 12\omega_{km}$. Then
$\omega_{k,l+1}>\frac 12\omega_{km}$ so $\omega_{l+1,m}<\frac
12\omega_{km}$. Then by the inductive hypothesis the claim holds for
$k,l$, i.e. $J_{kl}\leq L\omega_{kl}^{\gamma/p}$. Then $|I_{kl}^i|\leq
L\omega_{kl}^{\gamma/p}+B_2\omega_{kl}^{2/p}$ and $|y_l-y_k|\leq
L\omega_{kl}^{\gamma/p}+B_2\omega_{kl}^{2/p}+B_3\omega_{kl}^{1/p}$. Then
provided \Z\label{eq13}L\delta^{(\gamma-2)/p}\leq B_2\z we have
\Z\label{eq14}|I_{kl}^i|\leq 2B_2\omega_{kl}^{2/p}\ \ \ \ {\rm and}\ \ \ \ 
|y_l-y_k|\leq3B_2\omega_{kl}^{1/p}\z Putting these estimates into (\ref{eq12})
we find that
\[|J_{km}^i|\leq|J_{kl}^i|+|J_{lm}^i|+B_4\omega_{km}^{\gamma/p}\]
In the same way we find that
$|J_{lm}^i|\leq|J_{l,l+1}|+|J_{l+1,m}^i|+B_4\omega_{km}^{\gamma/p}$. 
Since $J_{l,l+1}=0$ we get
\[|J_{km}^i|\leq|J_{kl}^i|+|J_{l+1,m}^i|+2B_4\omega_{km}^{\gamma/p}\leq
L(\omega_{kl}^{\gamma/p}+\omega_{l+1,m}^{\gamma/p})+2B_4\omega_{km}
^{\gamma/p}\leq(2^{1-\gamma/p}L+2B_4)\omega_{km}^{\gamma/p}\]
and provided \Z\label{eq15}(1-2^{1-\gamma/p})L\geq 2B_4\z we conclude that
$|J_{km}^i|\leq L\omega_{km}^{\gamma/p}$ which completes the induction,
provided we choose $L$ and $\delta$ to satisfy (\ref{eq15}) and (\ref{eq13}),
and proves the claim.
 
As before, part (a) of the lemma follows from the claim.\vspace{.2cm}\\
Part (b) then follows from (a) in exactly the same way as for Lemma
\ref{l1}.
\end{proof}

Theorems \ref{th3} and \ref{th4} are deduced in exactly the same way as
Theorems \ref{th1} and \ref{th2} follow from Lemma \ref{l1}.\vspace{.2cm}\\
{\bf Remark 3.} The same reasoning as in Remark 1 shows that, if $f\in
C^{\gamma-1}$, then any solution of (1) on $[0,T]$ satisfies
\[y^i(t)-y^i(s)-f^i_j(y(s))(x^j(t)-x^j(s))-f^h_r(y(s))
\partial_hf^i_j(y(s))A^{rj}(s,t)=O(\omega(t)-\omega(s))^{\gamma/p}\]
and also that (\ref{eq9}) holds when $f\in C^\gamma$ in the present
situation.\vspace{.2cm}\\
{\bf Remark 4.} In the same way is in Remark 2, one can prove Theorem
\ref{th4} by showing directly that the discrete approximations converge,
without passing to subsequences.\vspace{.2cm}\\
{\bf Uniqueness when $\gamma=p$.}

Now we prove the slightly more delicate result that (\ref{eq1}) has a unique
solution when $f\in C^p$, where $2\leq p<3$ (the proof for $p<2$ is
similar but simpler).
We require the following lemma, whose proof is straightforward.

\begin{lemma}\label{l3} Suppose $a,b,c,d\in\R^n$ with $|a-b|<\lambda$,
$|c-d|<\lambda$, $|a-c|<\epsilon$ and $|a-b-c+d|<\sigma$.
Then\vspace{.2cm}\\
(a) if $F\in C^\gamma_0$, where $1\leq\gamma<2$, then 
\[|F(b)-F(a)-F(d)+F(c)|<C(\lambda^{\gamma-1}\epsilon+\sigma)\]
where $C$ depends only on the $C^\gamma$ norm of $f$.\vspace{.2cm}\\
(b) if $F\in C^\gamma_0$, where $2\leq\gamma<3$, then
\[|F(b)-F(a)-DF(a)(b-a)-F(d)+F(c)+DF(c)(d-c)|<
C\lambda(\lambda^{\gamma-2}\epsilon+\sigma)\]
where $C$ depends only on the $C^\gamma$ norm of $f$.
\end{lemma}

\begin{thm}\label{th5} Suppose $f\in C^p$ where $2\leq p<3$. Then the solution
of (\ref{eq1}), whose existence is asserted by Theorem \ref{th4}, is
unique.
\end{thm}
\begin{proof} Suppose $y(t)$ and $\tilde{y}(t)$ are two solutions; is
suffices to prove that $y=\tilde{y}$ on some interval $[0,\tau]$.
Suppose on the contrary that no such interval exists. Then for $k$ large
enough we can find $t_k>0$ such that $|y(t_k)-\tilde{y}(t_k)|=2^{-k}$
but $|y(t)-\tilde{y}(t)|<2^{-k}$ for $0<t<t_k$. Then
$t_k>t_{k+1}>\cdots$. We shall show that $\omega(t_k,t_{k+1})>$const$.
k^{-1}$. Since $\sum k^{-1}=\infty$, this will give a contradiction and
prove the theorem.

Since the problem is a local one, we can suppose $f\in C^p_0$. Fix $\gamma$
with $p<\gamma<3$. We use $C_1,C_2,\cdots$ for constants which depend
only on $p$, $\gamma$ and the $C^\gamma$-norm of $f$.

We fix $k$ and let $L$ denote the interval $[t_{k+1},t_k]$. We
introduce the notation
\[I^i(s,t)=y^i(t)-y^i(s)-f^i_j(x^j(t)-x^j(s))\]and
\[J^i(s,t)=I^i(s,t)-f^h_r(y(s))\partial_hf^i_j(y(s))A^{rj}(s,t)\]
We define $\tilde{I}$ and $\tilde{J}$ similarly.  Then by Remark 3 we have
$|J(s,t)|\leq C_1\omega(s,t)^{\gamma/p}$ and then
\Z\label{bds}|I(s,t)|\leq C_1\omega(s,t)^{2/p}\ \ {\rm and}\ \ |y(t)-y(s)|
\leq C_1\omega(s,t)^{1/p}\z
We also introduce the notation
$\y{y}(t)=\tilde{y}(t)-y(t)$, $\y{I}^i(s,t)=\tilde{I}^i(s,t)-I(s,t)$,
etc. Then we have $|\y{y}(s,t)|\leq2^{-k}$ if $[s,t]\subseteq L$ and
$|\y{J}(s,t)|\leq2C_1\omega(s,t)^{\gamma/p}$.

Now we write
\[R^i_j(s,t)=f^i_j(y(t))-f^i_j(y(s))-\partial_hf^i_j(y(s))(y^h(t)-y^h(s))\]
 Then
\[\begin{split}J^i(s,u)-J^i(s,t)-J^i(t,u)=&\left\{R^i_j(s,t)+\partial_h
f^i_j(y(s))I^h(s,t)\right\}(x^j(u)-x^j(t))\\&+\left\{g^i_{rj}(y(t))-g^i_{rj}
(y(s))\right\}A^{rj}(t,u)\end{split}\]
with a similar expression involving $\tilde{J}$.
The difference of the two expressions gives
\Z\label{eq16}\y{J}^i(s,u)-\y{J}^i(s,t)-\y{J}^i(t,u)=W^i_j(s,t)(x^j(u)-x^j(t))
+V^i_{rj}(s,t)A^{rj}(t,u)\z
where $V^i_{rj}(s,t)=g^i_{rj}(\tilde{y}(t))-g^i_{rj}(y(t))-g^i_{rj}
(\tilde{y}(s))+g^i_{rj}(y(s))$ and
\[W^i_j(s,t)=\y{R}^i_j(s,t)+\left\{\partial_hf^i_j(\tilde{y}(s))-\partial_h
f^i_j(y(s))\right\}I^h(s,t)+\partial_hf^i_j(\tilde{y}(s))\y{I}^h(s,t)\]

For each $n=1,2,\cdots$ let $K_n$ be the supremum of $|\y{J}(s,t)|$
taken over all intervals $[s,t]\subseteq L$ with $\omega(s,t)\leq2^{-n}$. Then
for such intervals $[s,t]$ we have the estimates $|\y{I}(s,t)|\leq
K_n+C_22^{-k-2n/p}$ and
\Z\label{eq17}|\y{y}(t)-\y{y}(s)|\leq K_n+C_22^{-k-n/p}\z
Using these estimates and Lemma \ref{l3} we obtain
$|V^i_{rj}(s,t)|\leq C_2(K_n+2^{-k-(p-2)n/p})$ and $|\y{R}^i_j(s,t)|\leq
C_22^{-n/p}(2^{-k-(p-2)n/p}+K_n)$. Now if $[s,u]$ is any interval in $L$
with $\omega(s,u)\leq2^{-n}$, we can find $t\in(s,u)$ with
$\omega(s,t)\leq2^{-n-1}$ and $\omega(t,u)\leq2^{-n-1}$. Then putting the
above estimates in (\ref{eq1}6) we obtain 
$|\y{J}^i(s,u)-\y{J}^i(s,t)-\y{J}^i(t,u)|\leq
C_3(2^{-n/p}K_{n+1}+2^{-n-k})$ and so
\Z\label{eq18}K_n\leq(2+C_32^{-n/p})K_{n+1}+C_32^{-n-k}\z
From the fact that $|\y{J}(s,t)|\leq2C_1\omega(s,t)^{\gamma/p}$ we
deduce that $K_n\leq2C_12^{-k-n}$ if $n>\frac{kp}{\gamma-p}$, and
combining this with the recurrence relation (\ref{eq18}) we see that $K_n\leq
C_4k2^{-k-n}$ for all $n$. Now if $n$ is such that $\omega(t_k,t_{k+1})
\leq2^{-n}$ then $2^{-k+1}\leq|\y{y}(t_k)-\y{y}(t_{k+1})|\leq
C_5(k2^{-k-n}+2^{-k-n/p})$ by (\ref{eq17}), so
$\omega(t_k,t_{k+1})\geq C_6k^{-1}$ as required.
\end{proof}
\noindent
{\bf Remark 5.} The hypothesis of finite $p$-variation on $x(t)$ can be
weakened slightly in Theorem \ref{th5} - it suffices to assume that
$|x^j(t)-x^j(s)|\leq\left(\omega(s,t)\log\log\frac1{\omega(s,t)}
\right)^{1/p}$ and $|A^{rj}(s,t)|\leq\left(\omega(s,t)\log\log\frac1
{\omega(s,t)}\right)^{2/p}$ for all sufficiently small intervals
$[s,t]$.

The proof of Theorem \ref{th5} can be repeated under these weaker assumptions.
To get the same bound for $J(s,t)$ using Remark 3, we can exploit the freedom
in the choice of $\gamma$ to use a slightly bigger $p$. Then in (\ref{bds})
there is an extra factor of $(\log\log\frac1{\omega(s,t})^{2/p}$ in the bound
for $I(s,t)$ and an extra factor of $(\log\log\frac1{\omega(s,t}^{1/p}$ in
the bound for $y(t)-y(s)$. This results in $2^{-n}$ being replaced by $2^{-n}
\log n$ in bounds such as (\ref{eq17}), the RHS of which becomes $K_n+C_22^
{-k-n/p}(\log n)^{1/p}$. Then in place of (\ref{eq18}) one has
\[K_n\leq(2+C2^{-n/p}\log n)K_{n+1}+C2^{-n-k}\log n\]
and deduces that $K_n\leq C'k\log k2^{-n}$. This gives
$\omega(t_k,t_{k+1})\geq\ $const$.(k\log k)^{-1}$, which, since $\sum(k\log
k)^{-1}=\infty$, is sufficient to complete the proof.

It is true that this is a very small improvement on Theorem \ref{th5}, but
in the case $p=2$, it is sufficient to make the theorem applicable
to Brownian motion - see Theorem \ref{th7}.
\vspace{.2cm}\\
{\bf Relation to results of \cite{tjl}}

Here we discuss briefly the relation of our notion of solution of (\ref{eq1}) to that given
by \cite{tjl} in the case $2\leq p<\gamma<3$. The basic object in \cite{tjl} is a `multiplicative
functional', which in effect consists of a path $x^i(t)$ together with iterated integrals
$A^{ij}(s,t)$ as considered in this paper. The viewpoint of \cite{tjl} is that the
solution should also be a multiplicative functional, so that the solution consists not
only of the path $y(t)$ as considered here but also of associated iterated integrals. Indeed
the solution as defined in \cite{tjl} includes the iterated integrals associated with the
path $(x(t),y(t))$ which combines the driving path with the solution path. This means, in
addition to $A^{ij}(s,t)$ which is given, also $B^{il}(s,t)$, $C^{kj}(s,t)$ and $D^{kl}(s,t)$ being
respectively interpretations of $\int_s^t(x^i-x^i(s))dy^l$, $\int_s^t(y^k-y^k(s))dx^j$ and $\int_s^t(y^k-y^k(s))dy^l$.

The inclusion of these additional components in the solution is the main difference between the notion of
solution in \cite{tjl} and ours, apart from the fact that \cite{tjl} defines solution in terms of an
integral equation formulation rather than a difference inequality like (\ref{eq10}). The relation between
the two notions is as follows. Assume $f\in C^{\gamma-1}$. Then for any solution of (\ref{eq1}) in the sense of Definition 4.1.1
of \cite{tjl}, it follows from the estimates in section 3.2 of \cite{tjl} that the path component $y$ will
satisfy (\ref{eq10}) and is therefore a solution in our sense. In the other direction, a solution of (\ref{eq2}) in
the sense of Definition \ref{def2} does not directly yield a solution in the sense of \cite{tjl}, because of the
missing components. But it can be shown that such a solution is obtained from any solution, in the sense of Definition \ref{def2},
of the extended system
\[dx^i=dx^i,\ dy^i=f^i_j(y)dx^j,\ dB^{il}=x^if^l_j(y)dx^j,\ dC^{kj}=y^kdx^j,\ dD^{kl}=y^kf^l_j(y)dx^j\]
(with the obvious initial conditions for $x$ and $y$, and arbitrary ones for $B,C,D$) by setting
\[B^{il}(s,t)=B^{il}(t)-B^{il}(s)-x^i(s)f^l_j(y(s))\{x^j(t)-x^j(s)\}\]
and similarly for the $C$ and $D$ terms.

\section{Equations driven by Brownian motion}\label{bm}

Stochastic differential equations driven by Brownian motion form one of
the main motivating examples for Lyons' theory. See \cite{iw} for background
on this topic. In this case the driving path is a $d$-dimensional Brownian
motion $W(t)=(W^1(t),\cdots,W^d(t))$ where $W^i(t)$ are independent standard
Brownian motions defined on $[0,\infty)$. Then with probability 1, $W$ has
finite $p$-variation for all $p>2$ on any finite interval, indeed it satisfies
a H\"{o}lder condition with exponent $1/p$, which means that we can take
$\omega(t)=t$ in the definition of $p$-variation. However $W$ does not have
finite 2-variation, so the theory of section \ref{s3} is needed.

There are two choices for $A^{rj}(s,t)$ corresponding to It\^{o} and
Stratonovich SDEs. For the It\^{o} case we use $A^{rj}_I(s,t)=\int_s^t
W^r(s,u)dW^j(u)$, where we use the notation $W^j(s,t)=W^j(t)-W^j(s)$,
and the integral is a standard It\^{o} integral. We define $A^{rj}_S(s,t)$
in the same way, but using a Stratonovich integral. The two versions differ
only on the diagonal, i.e. $A^{rj}_S=A^{rj}_I$ if $r\neq j$, and we have
$A^{jj}_S=A^{jj}_I+(t-s)/2=\frac12W^j(s,t)^2$. For either choice, with
probability 1 the $p/2$-variation condition with $\omega(t)=t$ is satisfied
on any finite interval for all $p>2$. We always assume that our Brownian
paths satisfy this condition, along with the $p$-variation condition for $W$
itself.

Then the theory described in section \ref{s3} here gives existence and
uniqueness of solutions for $f\in C^\gamma$ for $\gamma>2$ (As we shall see
in Theorem \ref{th7}, this can be improved to $f\in C^2$). On the
other hand, the well-established theory of SDE's gives existence and
uniqueness for any locally Lipschitz $f$, at least for It\^{o} equations. In
this section we attempt to account for this difference in smoothness
assumptions. We consider equation (\ref{eq1}) where $x(t)=W(t)$ is as
above.

In the first place, standard SDE theory regards a solution as a
stochastic process, and the uniqueness theorem gives uniqueness of a
process rather than uniqueness of a solution for an individual driving
path. However, we show in Proposition \ref{ie} that if $f\in C^\gamma$ for
$\gamma>1$, then with probability 1 the It\^{o} version of equation
(\ref{eq1}) has a unique solution, in the sense of Definition \ref{def2}). The real
reason for the difference in smoothness requirements is that the quantifiers
`for all $f\in C^\gamma$' and `with probability 1' do not commute. We shall
show that, if $\gamma<2$, the statement `with probability 1, for all $f\in
C^\gamma$, (\ref{eq1}) has a unique solution' is false (Theorem \ref{th6}
below). With $\gamma=2$, it is true (Theorem \ref{th7}).

We start by proving uniqueness of solutions for the It\^{o} version, for
given $f\in C^\gamma$ where $\gamma>1$.

\begin{lemma}\label{kb}
Consider the It\^{o} equation \Z\label{ieq} dy^i=f^i_j(y)dW^j\z with $y(0)=y_0$ on the interval
$[0,T]$ where $f\in C^\gamma_0$. Let $k\geq2$. Let $z^i(s,t)=y^i(t)-y^i(s)-f^i_j(y(s))(W^j(t)-W^j(s))-g^i_{rj}A^{rj}_I(s,t)$
where $g^i_{rj}(y)=f^h_r(y)\partial_hf^i_j(y)$. Then there is a constant $C$ such that
\[\IE|z(s,t)|^k\leq C(t-s)^{k(1+\gamma)/2}\]
\end{lemma}
\begin{proof}We use repeatedly the fact that if $X(t)$ is a stochastic
process adapted to the filtration of the Brownian motion, such that $\IE|X(t)|
^k\leq M$ for all $\tau$ in an interval $(s,t)$, then
\Z\label{ebd}\IE\left|\int_s^tX(t)dW^j(\tau)d\tau\right|^k\leq A(t-s)^{k/2}
M^k\z where $A$ is a constant depending on $k$.

First we have $y^i(t)-y^i(s)=\int_s^tf^i_j(y(\tau))dW^j(\tau)$ so $\IE|y(t)-y(s)|^k\leq C_1(t-s)^{k/2}$. Then
\[y^i(t)-y^i(s)-f^i_j(y(s))W^j(s,t)=\int_s^t\{f^i_j(y(\tau)-f^i_j(y(s))\}dW^j(\tau)\]
and so by (\ref{ebd})
\Z\label{p1}\begin{split}\IE|y^i(t)-y^i(s)-f^i_j(y(s))W^j(s,t)|^k&\leq C_2(t-s)^{k/2}\max_{s\leq\tau\leq t}
\IE|f(y(\tau))-f(y(s))|^k\\&\leq C_3(t-s)^k\end{split}\z
 Also we have
\[\IE|f^i(y(t))-f^i(y(s))-\partial f^i_j(y(s))(y^j(t)-y^j(s))|^k\leq C_4\IE|y(t)-y(s)|^{\gamma k}\leq C_5(t-s)^{\gamma k/2}\]
Combining this with (\ref{p1}) gives
\Z\label{p2}\IE|f^i_j(y(t))-f^i_j(y(s))-g^i_{rj}(y(s))W^r(s,t)|^k\leq C_5(t-s)^{\gamma k/2}\z
Finally
\[ z^i(s,t)=\int_s^t\{f^i_j(y(\tau))-f^i_j(y(s))-g^i_{rj}(y(s))W^r(s,\tau)\}dW^j(\tau)\]
and applying (\ref{ebd}) and (\ref{p2}) gives the required bound.
\end{proof}

Now we use the fact that, with probability 1, equation (\ref{ieq}) has a continuous {\em solution flow} $(s,t,x)\rightarrow F(s,t,x)
\in\R^d$, defined for $s<t$ and $x\in\R^d$, such that any choice of $s,t,x$ the solution of (\ref{ieq}) with $y(s)=x$ satisfies
$y(t)=F(s,t,x)$ with probability 1 (see \cite{ku}). Moreover, for any $\beta<\frac12$, $F(s,t,x)$ is a $C^\beta$ function of $s$ and $t$ and a locally
Lipschitz function of $x$, with uniform $C^\beta$ and Lipschitz bounds on compact sets. We define
\[Z^i(s,t,x)=F^i(s,t,x)-x^i-f^i_j(x)(W^j(t)-W^j(s))-g^i_{rj}A^{rj}(s,t)\]
and deduce that, if $0<\beta<\gamma-1$, then $Z$ is with probability 1 a $C^\beta$ function of $s,t,x$.

Now we can prove the following bound.

\begin{lemma}\label{zb}
Fix $T>0$, $L>0$ and $1<q<\A=(1+\gamma)/2$. Then with probability 1 there is a constant $C$ such that $|Z(s,t,y)|\leq C(t-s)^q$ for
$0\leq s<t\leq T$ and $|y|<L$.
\end{lemma}
\begin{proof} Fix $0<\beta<\min(\frac12,\gamma-1)$, and then fix $k$ large enough that $k(\A-q)>1+q(d+2)\beta^{-1}$. Next, for any
positive integer $N$, let $\eta_N=2^{-Nq/\beta}$ and let $\Omega_N$ be a finite set in $\R^d$ such that for any $y\in\R^d$ with $|y|<L$
one can find $y'\in\Omega_N$ with $|y-y'|<\eta_N$, and such that $\#(\Omega_N)\leq C_1\eta_N^{-d}$. Also let $\Lambda_N$ be a finite set
in $[0,T]$ with $\#(\Lambda_N)\leq T\eta_N^{-1}$ such that for any $t\in[0,T]$ there is $t'\in\Lambda_N$ with $|t-t'|<\eta_N$. Then for
any $\lambda>0$ we have, by lemma \ref{kb}, that
\[\begin{split}\IP(|Z(s,t,y)\geq\lambda2^{-Nq}\ {\rm for\ some}\ y\in\Omega_N\ {\rm and}\ s,t\in\Lambda_N\ {\rm with}\ t-s<2^{-N})
&\leq C_2\eta_N^{-d-2}\lambda^{-k}2^{-Nk(q-\A)}\\
&\leq C_2\lambda^{-k}2^{-N}\end{split}\]
It follows that with probability 1 there is $\lambda>0$ such that for every choice of $N\in\mathbb{N}$, $y\in\Omega_N$ and $s,t\in
\Lambda_N$ with $0<t-s<2^{-N}$ we have $|Z(s,t,y)|\leq\lambda2^{-Nq}$. It also follows from the above-mentioned $C^\beta$ property of
$Z$ that, with probability 1, there exists $B>0$ such that, for any $y,y',s,s',t,t'$ with $|y|<L$, $|y-y'|<\eta_N$, $|s-s'|<\eta_N$ and
$|t-t'|<\eta_N$ we have $|Z(s,t,y)-Z(s,t,y')|<B\eta_N^\beta=B2^{-Nq}$. Then given $y,s,t$ with $|y|<L $ and $t-s<2^{-N-1}$, we choose
$y'\in\Omega_N$ with $|y-y'|<\eta_N$, and $s',t'\in\Lambda_N$ with $|s-s'|<\eta_N$, $|t-t'|<\eta_N$, and conclude that $|Z(s,t,y)|\leq
(\lambda+B)2^{-Nq}$, and the required result follows.
\end{proof}

\begin{prop}\label{ie} 
Suppose $f\in C^\gamma$ where $\gamma>1$. Then, with probability 1, for any choice of $y_0$ the It\^{o} equation
$dy_i=f^i_jdW^j$ with $y(0)=y_0$ has either a solution in the sense of Definition \ref{def2} (with $\omega(t)=t$) for all $t\geq0$ or,
for some $T>0$, a solution on $0\leq t<T$ with $y(t)\rightarrow\infty$ as $t\rightarrow T$. Moreover the solution is unique in the sense
that if $\tilde{y}$ is any solution on $0\leq t<\tau$ in the sense of Definition \ref{def2} then $\tilde{y}=y$ for $0\leq t<\tau$.
\end{prop}
\begin{proof} For $n=1,2,\cdots$ let $f^{(n)}\in C^\gamma_0$ so that $f^{(n)}(y)=f(y)$ for $|y|\leq n$. Then we have the associated
flow $F^{(n)}(s,t,y)$ and $Z^{(n)}(s,t,y)$ defined as above for $f^{(n)}$ in place of $f$. By Lemma \ref{zb}, with probability 1
there is a sequence $(C_n)$ such that
\Z\label{zq}|Z^{(n)}(s,t,y)|\leq C_n(t-s)^q\z
whenever $n\in{\mathbb N}$, $0\leq s<t<n$ and $|y|<n$.
Using the Lipschitz property of the flow, we can also require that
\Z\label{fq}|F^{(n)}(s,t,x)-F^{(n)}(s,t,y)|\leq C_n|x-y|\z
whenever $n\in{\mathbb N}$, $0\leq s<t<n$ and $|x|,|y|<n$. We fix
a Brownian path $W$ for which these conditions hold, and prove the required existence and uniqueness for a solution driven by this path.

The existence of a solution $y$ is a consequence of Theorem \ref{th2}. To prove uniqueness, suppose $\tilde{y}$ is a solution
on $[0,\tau)$, with $\tau\leq T$, which is not identical to $y$ on $[0,\tau)$. Let $\tau_1=\sup\{t:\ t\geq0$ and $y(s)=\tilde{y}(s)$
for $0\leq s<t\}$. Then $0\leq\tau_1<\tau$ and $y(\tau_1)=\tilde{y}(\tau_1)$; we let $y_1=y(\tau_1)$. Now fix $\tau'$ with $\tau_1
<\tau'<\tau$ and choose $n>\tau$ such that $|y(t)|<n$ and $|\tilde{y}(t)|<n$ for all $t\in[\tau_1,\tau']$.\vspace{.2cm}\\
{\bf Claim:} $y(t)=F^{(n)}(\tau_1,t,y_1)$ for $\tau_1\leq t\leq\tau'$.

To prove this claim, fix $t\in[\tau_1,\tau']$ and let $N\in{\mathbb N}$. Let $\tau_1=t_0<t_1<\cdots<t_N=t$ with $t_{k+1}-t_k\leq N^{-1}$.
Let $v_k=y(t_k)$ and $w_k=F^{(n)}(\tau_k,t,v_k)$. Now from (\ref{eq10}) we have
\[|v^i_{k+1}-v^i_k-f^{(n)i}_j(v_k)(W^j_{k+1}-W^j_k)-g^{(n)i}_{rj}A^{rj}(t_k,t_{k+1})|\leq\theta(\tilde{\omega}(t_k,t_{k+1}))\]
where $\theta$ and $\tilde{\omega}$ are is in (\ref{eq10}). Together with (\ref{zq}) this gives
\[|v_{k+1}-u_k|\leq\theta(\tilde{\omega}(t_k,t_{k+1}))+C_nN^{-q}\]
where $u_k=F^{(n)}(t_k,t_{k+1},v_k)$. Then
\[|w_{k+1}-w_k|=|F^{(n)}(t_{k+1},t,v_{k+1})-F^{(n)}(t_{k+1},t,u_k)|\leq C_n\left(\theta(\tilde{\omega}(t_k,t_{k+1}))+C_nN^{-q}\right)\]
and so
\[|y(t)-F^{(n)}(\tau_1,t,y_1)|=|w_N-w_0|\leq C_n\left(\sum_k\theta(\tilde{\omega}(t_k,t_{k+1}))+C_nN^{1-q}\right)\]
which tends to 0 as $N\rightarrow\infty$, so we conclude that $y(t)=F^{(n)}(\tau_1,t,y_1)$ as claimed.

The same argument applies to $\tilde{y}$ and we conclude that $\tilde{y}=y$ on $[\tau_1,\tau']$, contradicting the definition
of $\tau_1$. This completes the proof of uniqueness.
\end{proof}

For Stratonovich equations the proof of Proposition \ref{ie} runs into difficulties because the existence of a solution flow
has not been proved in general, and the validity of the Proposition is an open question. In the case when the matrix of coefficients 
is nonsingular then it can be proved, using a standard type of change of variables which converts the equation to an It\^{o} equation, as
we now show.

\begin{lemma}\label{si}
Let $1<\gamma<2$, let $U$ be an open subset of $\R^n$, let $y$ be a solution of the Stratonovich equation $dy_i=f^i_j\circ dW^j$ in the
sense of Definition \ref{def2} on $\tau_1\leq t\leq\tau$, where $f\in C^\gamma$ and suppose $y([\tau_1,\tau])\subseteq U$. Suppose $\psi:
U\rightarrow\R^d$ is $C^{1+\gamma}$ and each component of $\psi$ satisfies $\sigma^{kh}(x)\partial_{kh}\psi^i(x)+\rho^k\partial_k
\psi^i(x)=0$ where
\Z\label{sr}\sigma_{kh}(y)=f^k_j(y)f^h_j(y)\ \ {\rm and}\ \ \rho^i(y)=f^k_j(y)\partial_kf^i_j(y)\z
Suppose also that $\tilde{f}$ is $C^\gamma$ on $\R^d$ and $\tilde{f}^i_j(\psi(y))=\partial_h\psi^i(y)f^h_j(y)$ for $y\in U$.

Then $x(t)=\psi(y(t))$ is a solution of the It\^{o} equation $dx^i=\tilde{f}^i_jdW^j$ on $[\tau_1,\tau]$ in the sense of Definition
\ref{def2}.
\end{lemma}
\begin{proof} Fix $\beta$ with $\frac13<\beta<\frac12$. By assumption
\[y^i(t)-y^i(s)=f^i_j(y(s))W^j(s,t)+g^i_{rj}(y(s))A^{rj}_S(s,t)+R^i(s,t)\]where $|R^i(s,t)|\leq
\theta(\tilde{\omega}_{st})$ (where $\theta$ and $\tilde{\omega}$ are as in (\ref{eq10})). Now $x(t)=\psi(x(t))$ and expanding $\psi(y)$
about $y=y(s)$ gives
\[\begin{split}x^i(t)=&\psi^i(y(s))+\partial_h\psi^i(y(s))f^h_j(y(s))W^j(s,t)+\frac12\partial_{hk}\psi^i(y(s))f^h_jf^k_rW^j(s,t)W^r(s,t)\\
&+\partial_h\psi^ig^h_{rj}A^{rj}_S(s,t)+O\left(\theta(\tilde{\omega}_{st})+(t-s)^{3\beta}\right)\end{split}\]
Now a calculation shows that
\[\frac12\partial_{hk}\psi^i(y(s))f^h_jf^k_rW^j(s,t)W^r(s,t)+\partial_h\psi^iG^h_{rj}A^{rj}_S(s,t)=\tilde{g}^i_{rj}(x)A^{rj}_I(s,t)\]
where $\tilde{g}^i_{rj}(x)=\tilde{f}^h_r(x)\partial_h\tilde{f}^i_j(x)$. Hence
\[x^i(t)=x^i(s)+\tilde{f}^i_j(x(s))W^j(s,t)+\tilde{g}^i_{rj}A^{rj}(s,t)+O\left(\theta(\tilde{\omega}_{st})+(t-s)^{3\beta}\right)\]
and the result follows.
\end{proof}

\begin{lemma}\label{sce}
Let $0<\alpha<1$ and consider the PDE
\Z\label{pse}\sigma^{kh}(y)\partial_{kh}\psi(y)+\rho^k\partial_k\psi(y)=0\z
where $\sigma$ is a matrix and $\rho$ a vector of $C^\A$ functions on a
neighbourhood of the origin in $\R^n$, such that $\sigma(0)$ is positive
definite. Then for any $\eta>0$, we can find a solution $\psi$ in $C^{2+
\alpha}$ on a neighbourhood of the origin, such that $|D\psi(0)-e_1|<\eta$,
where $e_1$ is the vector $(1,0,\cdots,0)$.
\end{lemma}
\begin{proof} We use the change of variable $y=\E x$ to obtain the equation $\sigma_{kh}(\E x)\partial_{kh}\phi(x)+\E\rho^k(\E x)
\partial_k\phi(x)=0$. By the Schauder theory (see \cite{gt}), for small $\E\geq0$ this equation has a unique $C^{2+\A}$ solution $\phi_\E$
in the unit ball satisfying $\phi_\E(y)=y^1$ on the boundary $\{|y|=1\}$, and $\phi_\E$ depends continuously on $\E$. Also $\phi_0(y)=y^1$
so $D\phi_0(0)=e_1$. Hence for small enough $\E>0$ we have $|D\phi_\E(0)-e_1|<\eta$ and then we can take $\psi(y)=\E\phi_\E(y/\E)$.
\end{proof}

\begin{prop}\label{se}
Suppose $f\in C^\gamma$ and $1<p<\gamma$. Suppose also that the matrix $(f^i_j(y))$ has rank $n$ for every $y\in\R^n$. Then, with
probability 1, for any choice of $y_0$ the Stratonovich equation
$dy_i=f^i_j\circ dW^j$ with $y(0)=y_0$ has either a solution in the sense of Definition \ref{def2} (with $\omega(t)=t$) for all $t\geq0$ or, for
some $T>0$, a solution on $0\leq t<T$ with $y(t)\rightarrow\infty$ as $t\rightarrow T$. Moreover the solution is unique in the sense
that if $\tilde{y}$ is any solution on $0\leq t<\tau$ in the sense of Definition \ref{def2} then $\tilde{y}=y$ for $0\leq t<\tau$.
\end{prop}
\begin{proof} Let $y_1\in\R^n$. Then by Lemma \ref{sce} we can find $C^{1+\gamma}$ functions $\psi^i$ for $i=1,\cdots,n$ satisfying
(\ref{pse}) with (\ref{sr}) and such that $\psi^i(y_1)$ is close to $e_i$, where $e_1,\cdots,e_n$ is the standard basis of $\R^n$. Then
$\psi=(\psi^1,\cdots,\psi^n)$ has non-zero Jacobian at $y_1$ and so is a diffeomorphism on a neighbourhood $U$ of $y_1$, so that we can find
$\tilde{f}\in C^\gamma_0$ such that $\tilde{f}^i_j(\psi(y))=\partial_h\psi^i(y)f^h_j(y)$ for $y\in U$.

Hence we can cover $\R^n$ by a sequence of open sets $(U_m)$ such that for each $m$ there is a $C^{1+\gamma}$ mapping $\psi^{(m)}: 
U\rightarrow\R^n$ satisfying (\ref{pse}) with (\ref{sr})and $f^{(m)}\in C^\gamma_0$ such that $f^{(m)i}_j(\psi^{(m)}(y))=\partial_h
\psi^{(m)i}(y)f^{(m)h}_j(y)$ for $y\in U_m$. To the It\^{o} equation $dy^i=f^{(m)i}_jdW^j$ we associate a solution flow and $Z^{(m)}
(x,s,t)$ as before, and then by Lemma \ref{zb}, with probability 1 there is a double sequence $(C_{rm})$ such that $|Z^{(m)}(s,t,x)|\leq
C_{rm}(t-s)^q$ whenever $r,m\in{\mathbb N}$, $0\leq s<t<r$ and $x\in\psi^{(m)}(U_m)$. We fix a Brownian path $W$ for which this holds,
and now prove uniqueness as in the It\^{o} case.

If we have two solutions $y$ and $\tilde{y}$ with the same initial condition which are not identical, then we define $\tau_1$ and $y_1$
just as in the proof of Proposition \ref{ie}. Then $y_1\in U_m$ for some $m$. Let $\tau>\tau_1$ be such that for $\tau_1\leq t\leq\tau$,
$y(t)$ and $\tilde{y}(t)$ are in $U_m$. Then $x(t)=\psi^{(m)}(y(t))$ and $\tilde{x}(t)=\psi^{(m)}(\tilde{y}(t))$ are, by Lemma
\ref{si} both solutions of the It\^{o} equation $dx_i=f^{(m)i}_j(x)dW^j$ on $[\tau_1,\tau]$ in the sense of Definition \ref{def2} and the proof is
concluded just as for Proposition \ref{ie}.
\end{proof}

We remark that versions of Propositions \ref{ie} and \ref{se} can be proved in the same way when $f$ is given to be $C^\gamma$ on an open
set $V$ in $\R^n$ and $y_0\in V$, and in the case of Proposition \ref{se} the matrix $(f^i_j(y))$ is assumed to be nonsingular for all
$y\in V$. Then we obtain a solution $y(t)\in V$ which is either defined for all $t>0$, or defined on $[0,T)$ and $y(t)$ leaves $V$ as
$t\rightarrow T$ in the sense that $y^{-1}([0,t])$ is a compact subset of $V$ for all $0\leq t<T$.

Next we show that uniqueness can fail for $f\in C^{2-\epsilon}$.

\begin{lemma}\label{l4} Let $W(t)$ denote standard Brownian motion in $\R^d$
where $d\geq5$. Let $0<\A<d-4$. Then for $\E>0$,
\[\IP({\rm dist}(W([0,2]),W([3,\infty)))<\E)<C\E^\A\]
where $C$ is a constant depending only on $d$ and $\A$.
\end{lemma}
\begin{proof} We use $c_1,c_2,\cdots$ for constants which depend only on
$d$ and $\A$. First we
make the observation that, for any given ball $B(a,r)$, we have
$\IP(W(t)\in B(a,r)$ for some $t\geq1)\leq c_1r^{d-2}$, which can be
verified by a straightforward calculation.

Let $\gamma=(d-2-\A)^{-1}<\frac12$ and fix a positive integer $N$. We write
$t_j=\frac jN$. Then, for $j=0,1,\cdots,2N-1$ we have
$\IP(W([t_j,t_{j+1}])\not\subseteq B(W(t_j),N^{-\gamma}))\leq
c_2e^{-\frac12N^{1-2\gamma}}$ and so
\[\IP({\rm for\ all}\  j=0,1,\cdots,2N-1\ {\rm we\ have}
\ W([t_j,t_{j+1}])\subseteq B(W(t_j),N^{-\gamma}))\geq1-2c_2Ne^{-\frac12
N^{1-2\gamma}}\]
Also, by the above
observation (applied to $W$ starting at time 2), we have for each
$j=0,1,\cdots,2N-1$ that
\[\IP(W([3,\infty))\ {\rm meets}\ B(W(t_j),2N^{-\gamma}))\leq c_3N^{-(d-2)\gamma}\] and so
\[\IP(W([3,\infty))\ {\rm avoids}\ B(W(t_j),2N^{-\gamma})
 \ {\rm for}\ j=0,1,\cdots,2N-1)\geq1-2c_3N^{1-(d-2)\gamma}\]
Putting these facts together we obtain
\[\IP({\rm dist}(W([0,2]),W([3,\infty)))<2N^{-\gamma})<c_3N^{1-(d-2)\gamma}+
2Nc_2e^{-\frac12N^{1-2\gamma}}<c_4N^{1-(d-2)\gamma}\] and the lemma
follows on choosing $N$ so that $\E\approx N^{-\gamma}$.
\end{proof}

\begin{thm}\label{th6} Let $\epsilon>0$ and let $W(t)$ be standard Brownian
motion on $\R^6$. Then, with probability 1, there exists a compactly
supported $C^{2-\epsilon}$ function $f$ on $\R^6$, such that the system
\[dy^i=dW^i,\ i=1,\cdots,5,\ \ \ \ \ dy^6=f(y)dW^6\]
has infinitely many solutions, in the sense of Definition \ref{def2}, satisfying the initial condition
$y(0)=0$.
\end{thm}
\begin{proof} We use $b_1,b_2,\cdots$ to denote positive constants which
can depend on $\epsilon$ but on nothing else. We write $\eta=\epsilon/3$ and
introduce the notation $I_n=[(n+1)^{-\eta},n^{-\eta}]$,
$\tilde{I}_n=[(n+\frac43)^{-\eta},(n-\frac13)^{-\eta}]$. We also write
$\rho_k=k^{-4}2^{-k(1+\eta)/2}$ for $k=1,2,\cdots$. We write
$W^*(t)=(W^1(t),\cdots,W^5(t))$.

For a given path $W^*(t)$, we define $\Omega_k$ to be the set of odd
integers $n$ with $2^k\leq n<2^{k+1}$ and
dist$(W^*(I_n),W^*([0,1]\backslash\tilde{I}_n))\geq\rho_k$. Using Lemma 4,
and the fact that $|I_n|\sim\eta n^{-1-\eta}$, we see that for a given odd
$n$ with $2^k\leq n<2^{k+1}$, we have $\IP(n\not\in\Omega_k)\leq
b_1k^{-2}$.

For $n\in\Omega_k$ we find $f_n$ in $C^2(\R^5)$ such that $f_n(x)=1$ if
$x\in W^*(I_n)$ and $f_(x)=0$ if dist$(x,W^*(I_n))\geq\rho_k$, and such that
$\|f_n\|_{C^\alpha}\leq b_2\rho_k^\alpha$ for $0<\alpha\leq2$. Note that
then $f_n(W^*(t))=0$ for $t\in[0,1]\backslash\tilde{I}_n$. Now let
$\alpha_n=\int f_n(W^*(t))dW^6(t)$; then (for a fixed path $W^*$),
$\alpha_n$ is normally distributed with mean 0 and variance $\int f_n(W^*(t))
^2dt\geq|I_n|$, so $\IP(|\alpha_n|\leq k^{-2}2^{-k(1+\eta)/2})\leq b_3k^{-2}$.

Now we define $V_k$ to be the set of odd integers $n$ with $2^k\leq
n<2^{k+1}$ such that either $n\not\in\Omega_k$, or $n\in\Omega_k$ and
$|\alpha_n|\leq k^{-2}2^{-k(1+\eta)/2}$. Then for each odd $n$ with
$2^k\leq n<2^{k+1}$ we have $\IP(n\in V_k)\leq b_4k^{-2}$. So if $X_k$ is
the cardinality of $V_k$ then $\IE X_k\leq b_4k^{-2}2^{k-1}$ and so
$\IP(X_k\geq2^{k-2})\leq2b_4k^{-2}$. Hence almost surely there is $k_0$
such that for $k>k_0$ we have $X_k\leq2^{k-2}$, which implies
$\sum_{n\in\Omega_k}|\alpha_n|\geq\frac14k^{-2}2^{k(1-\eta)/2}$. Now
let $\sigma_n=$sign$(\alpha_n)$; then we have, for $k>k_0$, that
$\sum_{n\in\Omega_k}\sigma_n\alpha_n\geq\frac14k^{-2}2^{k(1-\eta)/2}$.

We also have that for $n\in\Omega_k$,
$\IP(\max_t\left|\int_t^1f_n(W^*(s))dW^6(s)\right|>1)\leq
b_5e^{-b_6n}$, and so with probability 1 we can (redefining $k_0$ if
necessary) suppose that $\left|\int_t^1f_n(W^*(s))dW^6(s)\right|\leq1$
for $n\in\Omega_k$, when $k>k_0$.

Let $\phi$ be a smooth function of one variable, vanishing outside
$[-2,2]$, such that $\phi(x)=x$ for $|x|\leq1$. Then define $f_k$ on
$\R^6$ by $f_k(x)=\rho_k^{2-\epsilon}\phi(x^6/\rho_k)\sum_{n\in\Omega_k}
\sigma_nf_n(x_1,\cdots,x_5)$, and set $f(x)=\sum f_k(x)$. Next,
let $\psi(t)=\sum_k\sum_{n\in\Omega_k}\psi_n^{\sigma_n}(t)$ where
\[\psi_n^\sigma(t)=-\rho_k^{1-\epsilon}\sigma\int_t^1f_n(W^*(s))dW^6+\frac12
\rho^{2(1-\epsilon)}\int_t^1f_n(W^*(s))^2ds\]
for $\sigma=\pm1$.

Now for $t<2^{-(k+1)\eta}$, where $k>k_0$, we have
\[\rho_k^{1-\epsilon}\sum_{n\in\Omega_k}\int_t^1f_n(W^*(s))dW^6(s)=
\rho_k^{1-\epsilon}\sum_{n\in\Omega_k}|\alpha_n|\geq\frac14
\rho_k^{1-\epsilon}k^{-2}2^{k(1-\eta)/2}\geq
k^{-6}2^{k(\epsilon+\eta\epsilon-2\eta)/2}\geq2^{k\eta/2}\]
for $k$ large enough. Note also that for all $n$, $\int_t^1f_n(W^*)dW^6$ is
either $\alpha_n$ or 0, unless $t\in\tilde{I}_n$, which for a given
$t$ can occur for at most 2 values of $n$. Taking into account the fact
that $\left|\int_t^1f_n(W^*(s))dW^6(s)\right|\leq1$, and noting that
the second term in the expression for $\psi(t)$ is bounded by 1 in
absolute value, it follows that $\psi(t)\leq-bt^{-1/2}$ for $t$ small.

Now let $y(t)=A\exp\psi(t)$, where $A$ is a constant. Then
\Z\label{dy}dy(t)=\sum_k\rho_k^{1-\epsilon}\sum_{n\in\Omega_k}\sigma_nf_n
(W^*(t))y(t)dW^6(t)\z
and if $|A|$ is small enough, $|y(t)|<\rho_k$ whenever any
$f_n(W^*(t))$ is non-zero, so (\ref{dy}) can be written as
$dy(t)=f(W(t))dW^6(t)$.

It follows that the system in the statement of
the theorem has solution $y^i=W^i$, $i=1,\cdots,5$, $y^6(t)=A\exp\psi(t)$
for any sufficiently small $A$, provided we can verify that (\ref{dy}) holds
in the sense of (\ref{eq10}). This can be done easily as follows: for $2^k
\leq n<2^{k+1}$ and $\sigma=\pm1$ let $y_n^\sigma(t)=\exp\psi_n^\sigma(t)$
and then
\[dy_n^\sigma=\rho_k^{1-\epsilon}\sigma f_n(W^*(t))y_n^\sigma(t)dW^6(t)\]
and then from H\"{o}lder bounds for $W$ we can deduce that
\Z\label{yn}y_n^\sigma(t)-y_n^\sigma(s)-\rho_k^{1-\epsilon}\sigma f_n(W^*(s))y_n^\sigma(s)\{W^6(t)-W^6(s)
\}-\rho_k^{1-\epsilon}\sigma\sum_{j=1}^5\partial_jf_n(W^*(s))A^{j6}_I(s,t)\z
is bounded in absolute value by $C_k(t-s)^\gamma$
where $C_k$ is dominated by a suitable power of $2^k$. Now note that on $\tilde
{I}_n$, $y$ is $y_n^{\sigma_n}$ multiplied by a positive constant which is $\leq A\exp
(-b2^{k\eta/2})$. This rapid exponential decay as $t\rightarrow0$ means that
the required bound for $y$ follows easily from the above bound for (\ref{yn}).
\end{proof}

Next we apply Theorem \ref{th5}, with the improvement described in Remark
5, to show that with probability 1 uniqueness holds for every $f$ in $C^2$.
The proof requires a variant of the Law of the Iterated Logarithm. To state
this we introduce the following notation: given $\tau\geq0$, let
\[M(\tau)=\max\frac{\sum_i|W^i(t)-W^i(s)|^2+\sum_{r,j}
|A^{rj}(s,t)|}{(t-s)\log\log(t-s)}\]
where the max is over all $s,t$ with $0\leq s\leq\tau\leq t\leq T$ and
$t-s\leq\frac1{10}$. Then we have

\begin{lemma}\label{lil} For any $T>0$ there are constants $c_1$ and $c_2$
such that $\IP(M(\tau)\geq K)\leq c_1e^{-c_2K}$ for any $\tau\in(0,T)$ and
$K>0$.
\end{lemma}
\begin{proof} For any interval $I=(s,t)$ we write $X_I=\sum_i|x^i(t)-x^i(s)|
^2+\sum_{r,j}|A^{rj}(s,t)|$. Then if $I$ has length $2^{-k}$ we have
$\IP(X_I\geq\mu2^{-k})\leq C_1e^{-C_2\mu}$ for all $\mu>0$. Let $\lambda>1$.
Then the probability that, for some $k>1$ and some dyadic $I\subseteq[0,T]$ with length $2^{-k}$,
we have $X_I>\lambda2^{-k/2}(|\tau-s|+2^{-k})^{1/2}$, does not exceed
\[C_1\sum_{k=2}^\infty\sum_I\exp(-C_2\lambda\log k(1+2^k|\tau-s|)^{1/2}\leq
C_3\sum_{k=1}^\infty k^{-C_2\lambda}\leq C_42^{-C_2\lambda}\]
where $\sum_I$ denotes a sum over all dyadic intervals of length $2^{-K}$ in $[0,T]$.
Hence, with probability at least $1-C_42^{-C_2\lambda}$ we have
\Z\label{xbd}X_I\leq\lambda2^{-k/2}(|\tau-s|+2^{-k})^{1/2}\z
for all $k\geq2$ and dyadic intervals $I\subseteq[0,T]$ with length $2^{-k}$.
Now if $I=(s,t)$ is any subinterval of $[0,T]$ containing $\tau$, then we can
express $I$ as the union of non-overlapping dyadic intervals, such that not
more then 2 of them can have the same length. Then when (\ref{xbd}) holds we
obtain $X_I\leq C_5\lambda(t-s)\log\log\frac1{t-s}$. Then $M(\tau)\leq C_5
\lambda$, and this holds with probability at least $1-C_42^{-C_2\lambda}$,
which gives the result.
\end{proof}

\begin{thm}\label{th7} If $W^i(t)$ are independent Brownian motions, then
with probability 1, for all $f\in C^2$ the equation (\ref{eq1}), with
$x^i=W^i$, has a unique solution in the sense of Definition \ref{def2}.
\end{thm}
\begin{proof} We work on a fixed interval $[0,T]$. By Remark 5, it suffices
to show that, with probability 1, there is an increasing function $\omega(t)$
on $[0,T]$ such that
$|W^j(t)-W^j(s)|\leq\omega(s,t)^{1/2}\left(\log\log\frac1{\omega(s,t)}
\right)^{1/2}$ and $|A^{rj}(s,t)|\leq\omega(s,t)
\log\log\frac1{\omega(s,t)}$ for all sufficiently small intervals $[s,t]$.

To do this, we apply Lemma \ref{lil} to assert that $\IE M(\tau)^2\leq C_3$
for each $\tau\in[0,T]$. Then $\IE\int_0^TM(\tau)^2<\infty$ so
with probability 1, $\int_0^TM(\tau)^2d\tau<\infty$. When this integral is
finite we can define $\omega(t)=\int_0^tM(\tau)d\tau$ and note that by
Cauchy-Schwartz $\omega(s,t)\leq C(t-s)^{1/2}$. Then
$|W^i(t)-W^i(s)|^2\leq\omega(s,t)\log\log\frac1{t-s}\leq
C'\omega(s,t)\log\log\frac1{\omega(s,t)}$ with a similar bound for
$A^{rj}(s,t)$, which completes the proof.
\end{proof}

We conclude this section with an example showing that continuous
differentiability of $f$ is not sufficient for (even local) existence.
The construction, which is similar to Theorem \ref{th6} is based on the
following lemmas.

\begin{lemma}\label{l5} Let $W(t)$ be standard Brownian motion in $\R^d$ where
$d\geq5$, let $\gamma>\frac12+\frac1{d-4}$ and suppose $\A$ satisfies
$(\gamma-\frac12)^{-1}<\A<d-4$. Let $\eta>0$. Then with probability at
least $1-C\eta^\A$ we have that
\[|W(s)-W(t)\geq\eta|s-t|^\gamma\]
for all $s,t\in[0,1]$.
\end{lemma}
\begin{proof}  For
integers $r,k\geq0$ let $E_{rk}$ denote the event dist$(W(k2^{-r},
(k+1)2^{-r}),W((k+2)2^{-r},\infty))\leq\eta2^{-r\gamma}$. Then by Lemma \ref{l4},
$\IP(E_{rk})\leq C_1\eta^\A2^{-r(\gamma-\frac12)\A}$ and so, writing
$\delta=(\gamma-\frac12)\A-1$, we have $\IP(\cup_{k=0}^{2^r-1}E_{rk})\leq
C_1\eta^\A2^{-r\delta}$. The result follows by summing over $r$.
\end{proof}

\begin{lemma}\label{l6} Let $M>0$ and let $W(t)$ be standard Brownian motion
on $\mathbb{R}^8$. Then with probability 1 we can find a compactly
supported smooth function $f$ on $\mathbb{R}^7$ such that
$\sup|f|\leq1$, $\sup|Df|\leq1$ and $\int_0^1f(W^*(t))dW_8(t)>M$, where
$W^*=(W_1,\cdots,W_7)$.
\end{lemma}
\begin{proof} Fix $\gamma$ with $\frac56<\gamma<1$ and then choose
$\alpha$ so that $(\gamma-\frac12)^{-1}<\A<3$. Note that then $\A>1$, so
we can fix $\beta$ with $0<\beta<1$ and $\alpha\beta>1$.

Let $k$ be a positive integer. For $n=1,2,\cdots,2^k$ let
$I_{kn}=[(n-1)2^{-k},n2^{-k}]$ and let
$\tilde{I}_{kn}=((n-2)2^{-k},(n+1)2^{-k})$. Given a path $W^*$, let
$\Omega_k$ be the set of odd integers $n$ with $0<n<2^k$ and
dist$(W^*(I_{kn}),W^*([0,1]\backslash\tilde{I}_{kn}))\geq\rho_k$, where
$\rho_k=k^{-\beta}2^{-k/2}$, and such that also
\[|W^*(s)-W^*(t)|\geq k^{-\beta}2^{(\gamma-\frac12)k}|s-t|^\gamma\]
for all $s,t\in\tilde{I}_{kn}$.  Let $N_k$ be the cardinality of
$\Omega_k$. By Lemmas \ref{l4} and \ref{l5}, with $d=7$ and scaling of $t$,
we see that for any odd $n$ we have $\IP(n\notin\Omega_k)\leq
C_1k^{-\alpha\beta}$ so $\IE(2^{k-1}-N_k)\leq
C_12^{k-1}k^{-\alpha\beta}$ and hence $\IP(N_k\leq2^{k-2})\leq
2C_1k^{-\alpha\beta}$. It follows that, with probability 1, there exists
$k_0$ such that $N_k>2^{k-2}$ for all $k\geq k_0$.

Still considering a fixed path $W_*$, with $k\geq k_0$, we find for each
$n\in\Omega_k$ a function $g_{kn}$ on [0,1] such that
$0\leq g_{kn}\leq\rho_k$ everywhere, $g_{kn}=0$ outside $I_{kn}$,
$|g_{kn}(s)-g_{kn}(t)|\leq2\rho_k|s-t|$ for all $s,t$, and $\int
g_{kn}^2=\frac13\rho_k^22^{-k}$. Let $F=W^*([0,1])$ and define $f_{kn}$
on $F$ by $f_{kn}(W^*(t))=g_{kn}(t)$ and note that from the definition
of $\Omega_k$ we have
\Z\label{eq19}|f_{kn}(x)-f_{kn}(y)|\leq C_2\rho_k\min\left(1,\left\{\frac{|x-y|}
{\rho_k}\right\}^{1/\gamma}\right)\z
for all $x,y\in F$.

Now let $\alpha_{kn}=\int f_{kn}(W^*(t))dW_8(t)=\int g_{kn}(t)dW_8(t)$.
Conditional on $W^*$, for fixed $k$ the $\A_{kn}$ are independent
normally distributed random variables with mean 0, and Var$(\A_{kn})=
\frac132^{-k}\rho_k^2$. Now let $X_k=\sum_{n\in\Omega_k}|\A_{kn}|$.
Then, using $\rho_k=2^{-k/2}k^{-\beta}$, we obtain $\IE X_k=
\sqrt{2/3\pi}N2^{-k}k^{-\beta}$ and Var$(X_k)=\frac13N_k2^{-2k}k^{-2\beta}$.
Then by Chebychev's theorem $\IP(X_k\leq N_k2^{-k-1}k^{-\beta})\leq C_3N_k^{-1}$.
Since $N_k\geq2^{k-2}$ we deduce $\IP(X_k\leq\frac18k^{-\beta})\leq
C_42^{-k}$. It follows that with probability 1 we have
$\sum_{k=k_0}^\infty X_k=\infty$, so we can find $k_1$ so that
$\sum_{k=k_0}^{k_1}X_k>M$.

We now need to extend $f_{kn}$ to the whole of $\R^7$ and smooth it. To
this end, we use Whitney's extension theorem (see Section VI.2 of \cite{st})
which gives a bounded linear mapping $T$ from the space of Lipschitz
functions on $F$ to the Lipschitz functions on $\R^7$. We also let $\phi\in
C_0^\infty(\R^7)$ with $\int\phi=1$, set $\phi_\E(x)=\E^{-7}\phi(x/\E)$, and
let $f_{kn}^\E=\phi_\E*Tf_{kn}$ for $\E>0$. Let $\A_{kn}^\E=\int
f_{kn}^\E(W^*(t))dW_8(t)$. Then with probability 1, $\A_{kn}^\E
\rightarrow\A_{kn}$ as $\E\rightarrow0$. So if $\E$ is chosen small
enough, we have $\sum_{k=k_0}^{k_1}\sum_n|\A_{kn}^\E|>M$. We fix such an
$\E$ and let $\sigma_{kn}=$sign$(\A_{kn}^\E)$. Now if
$h=\sum_{k=k_0}^{k_1}\sum_n\sigma_{kn}f_{kn}$ then (\ref{eq19}) implies a
Lipschitz bound $|h(x)-h(y)|\leq C_5|x-y|$ for $x,y\in F$. Now let $f=
\sum_{k=k_0}^{k_1}\sum_n\sigma_{kn}f_{kn}^\E=\phi_\E*Th$. Then $f$ is
smooth and $|Df|\leq C_5$ everywhere. Moreover $\int
f(W^*(t))dW_8(t)=\sum_{k=k_0}^{k_1}|\A_{kn}^\E|>M$, completing the
proof.
\end{proof}

\begin{thm}\label{th8} Let $W(t)$ be standard Brownian motion on $\R^8$. Then
with probability 1 there exists a compactly supported continuously
differentiable function $f$ on $\R^7$, which is $C^\infty$ on
$\R^7\backslash\{0\}$, such that the system
\[dy_i=dW_i,\ i=1,\cdots,7,\ \ \ \ \ dy_8=f(y_1,\cdots,y_7)dW_8\]
has no solution satisfying $y(0)=0$. To be more precise, there is no
continuous $y(t)$ on [0,1], such that $y(0)=0$ and the above equation is
satisfied locally on $(0,1)$ in the sense of Definition \ref{def2}, this notion being
well-defined since, with probability 1, $W^*$ avoids the origin for
$t>0$, and $f$ is smooth away from 0.
\end{thm}
\begin{proof} For $k$ even and nonnegative let $I_k=[2^{-k-1},2^{-k}]$.
With probability 1 the intervals $W^*(I_k)$ are disjoint so we can find
a sequence of smooth functions $\psi_k$ with disjoint compact supports
($k=0,2,4,\cdots$) such that $\psi_k=1$ on $W^*(I_k)$. By Lemma \ref{l6} we
can find smooth $f_k$ such that $\|f_k\|_{C^1}\|\psi_k\|_{C^1}\leq k^{-2}$
and $\int_{I_k}f_k(W^*(t))dt>1$. Let $f=\sum f_k\psi_k$; then $f$ is
$C^1$ and $\int_{I_k}f_k(W^*(t))dt>1$ for even $k$. For this $f$, any
solution to the system must satisfy $W_8(2^{-k})-W_8(2^{-k-1})>1$ for
all even $k$, and so cannot be continuous at 0.
\end{proof}

One may expect that the dimensions of the spaces in Theorems \ref{th6} and \ref{th8} could
be considerably reduced. The point of the high-dimensional Brownian paths is to give good
separation between segments of the path, which avoids technical problems in the proofs.
Constructions in lower dimensions would probably be more complicated.

We remark that rough path theory can be used to interpret anticipating stochastic differential
equations of the form $dy^i=f^i_j(y)dW^j$ where $f^i_j$ is random in the sense that
it depends on the path $W$, without any adaptedness condition, provided $f$ has, with probability 1, the required
smoothness w.r.t. $y$ for the theory to apply. Theorem \ref{th6} and the results following it can be interpreted
in this light. Thus when $f$ is almost surely $C^2$ as a function of $y$, Theorem \ref{th7} asserts the
existence of a unique solution, with this interpretation. The proofs of Theorems \ref{th6} and \ref{th8},
in which $f$ is constructed given the path $W$, can easily be modified so that $f$ depends measurably on $W$,
and give counterexamples in this setting.

Other interpretations of anticipating SDEs can be found for example in \cite{np}. See \cite{fr} for a recent study of
the relation of the rough path approach to such other approaches.

\section{Other examples}\label{oe}

The examples below indicate that the smoothness requirements on $f$ in
the results of sections \ref{s2} and \ref{s3} are sharp in respect of the inequalities relating $\gamma$ and $p$.\vspace{.2cm}\\
{\bf Example 1.} Nonuniqueness of solutions for $f\in C^\gamma$ when
$1<\gamma<p<2$.

Suppose $1<\gamma<p<2$. Let $\beta$ and $\rho$ be large positive numbers
with $\gamma<\frac\rho\beta<\frac{\rho+1}\beta<p$, and let
$\alpha=p^{-1}$. Let $x^1(t)=t^\beta\cos(t^{-\rho})$,
$x^2(t)=t^\beta(2+\sin(t^{-\rho}))$.
Then $x^i\in C^\alpha$ since $\alpha<\beta/(\rho+1)$.
Next, we can find a $C^\gamma$ function $f$ such that
$f(y^1,y^2)=(y^2)^\gamma$ if $|y^1|>y^2>0$ and it is 0 if $y^1=0$.
Then the system
\[dy^1=f(y^1,y^2)dx^1,\ \ \ \ \ dy^2=dx^2,\ \ \ \ \  y^i(0)=0\]
has two solutions in $C^\alpha$ for small $t\geq0$:\\
$y^2=x_2$, $y^1=0$ and $y^2=x^2$, $y^1=\int(x^2)^\gamma dx^1$.

To verify the second solution, one needs to check that $y^1\geq{\rm
const}\ t^{\beta(\gamma+1)-\rho}\geq3t^\beta\geq x^2$ for $t$ small.
\vspace{.2cm}\\
{\bf Example 2.} Nonuniqueness of solutions for $f\in C^\gamma$ when
$2<\gamma<p<3$.

When $2<\gamma<p<3$ we can use the same construction as in example 1,
again with $\gamma<\frac\rho\beta<\frac{\rho+1}\beta<p$. Again we get
the same solutions as above, provided we interpret the differential
equation naively (everything being smooth for $t>0$). However this does
not fit in with the theory in section 3, because it requires
$A^{ij}(s,t)=\int_s^t\{x^i(u)-x^i(s)\}dx^j(u)$ (interpreting the integrals
naively), which does not satisfy the $p/2$-variation requirement.

One can get round this problem by defining
$A^{ij}(s,t)=-x^i(s)\{x^j(t)-x^j(s)\}$. One can check that this
satisfies the consistency condition and the variation requirement, and
that then both choices of $y^1,y^2$ are solutions, in the sense of Definition \ref{def2} of the modified
system
\[dy^1=(1-\rho)f(y^1,y^2)dx^1,\ \ \ \ \ dy^2=dx^2,\ \ \ \ \  y^i(0)=0\]
{\bf Example 3.} Nonexistence of solutions for $f\in C^{p-1}$ when
$1<p<2$.

Let $1<p<2$ and let $\alpha=1/p$. For $k=1,2,\cdots$ let $n_k$ be the
smallest integer $\geq2^{k-1}/(k\pi)$ and let $t_k=\pi n_k2^{1-k}$; then
$0<t_k\leq\pi$ and $t_k\sim1/k$ for $k$ large.

For $t\in[0,\pi]$ let $x^1(t)=\sum2^{-\alpha k}\sin(2^kt)$ where the sum
is over those integers $k\geq1$ with $t_k\geq t$. Then $x^1\in
C^\alpha$, and is locally Lipschitz on $(0,\pi]$. Also define
$z(t)=\sum_{k=1}^\infty2^{-(1-\alpha)k}\cos(2^k)$. Then $z\in
C^{1-\alpha}$. Now, using Lemma \ref{l7} below, we can find $x^2$ and $x^3$ in
$C^\alpha$ such that
$|(x^2(s),x^3(s))-(x^2(t),x^3(t))|\geq$const$|s-t|^\alpha$. Then, using
Whitney's extension theorem, we can
write $z(t)=f(x^2(t),x^3(t)$ where $f\in C^{p-1}$.

Now consider the system
\[dy^1=f(y^2,y^3)dx^1,\ \ dy^2=dx^2,\ \ dy^3=dx^3;\ \ \ \ y^1(0)=0,\ \
y^2(0)=x^2(0),\ \ y^3(0)=x^3(0)\]
Suppose we have a solution (in the sense of Definition \ref{def1}) of this system on an
interval $[0,\tau]$, where $0<\tau<\pi$. Then we must have $y^2=x^2$,
$y^3=x^3$, and, since $x^1$ is locally Lipschitz for $t>0$, the equation
can be interpreted naively for $t>0$, and we have, for
any $0<s<\tau$, that
\[\begin{split}y^1(\tau)-y^1(s)&=\int_s^\tau
zdx^1=\sum_k'\sum_{l=1}^\infty2^{(1-\alpha)(k-l)}\int_s^{\tau_k}
\cos(2^kt)\cos(2^lt)dt\\&=\frac12\sum_k'\sum_{l=1}^\infty2^{(1-\alpha)(k-l)}
\frac{\sin(2^k+2^l)\tau_k-\sin(2^k+2^l)s}{2^k+2^l}+\frac12\sum_k'(\tau_k-s)\\
&\ \ +\frac12\sum_k'\sum_{l\neq k}2^{(1-\alpha)(k-l)}
\frac{\sin(2^k-2^l)\tau_k-\sin(2^k-2^l)s}{2^k-2^l}\\&=\frac12\log\frac1s+O(1)\end{split}\]
as $s\rightarrow0$. Here $\tau_k=\min(\tau,t_k)$ and $\sum_k'$ denotes a
sum over those $k$ for which $t_k>s$. But then $y^1(s)\rightarrow\infty$
as $s\rightarrow0$, so (\ref{eq2}) is not satisfied at 0.

Hence no solution exists on any interval $[0,\tau]$.

The above proof used the following (probably known) lemma:

\begin{lemma}\label{l7} Suppose $\frac12<\alpha<1$. Then we can find positive
constants $c_1$ and $c_2$, and a function $u$ on [0,1] taking values in
$\R^2$, such that
\[c_1|s-t|^\alpha\leq|u(s)-u(t)|\leq c_2|s-t|^\alpha\]
for all $s,t\in[0,1]$.
\end{lemma}
\begin{proof} We shall use the following terminology: given a lattice of
squares of side $\epsilon$, a {\em chain of squares} of side
$\epsilon$ is a sequence $Q_1,\cdots,Q_n$ of squares in the lattice,
such that $Q_i$ and $Q_{i+1}$ have one side in common, $Q_i$ and $Q_j$
are disjoint if $|i-j|>2$ and have at most a corner in common if
$|i-j|=2$.

Now, since $\frac12<\alpha<1$, it is not hard to construct bounded
sequences of integers $k_r$ and $m_r$, such that $k_r\geq2$, $m_r$ is
odd, $n_r\leq m_r\leq k_r^2$ where $n_r=2k_r+1$, and such that the sequence
$\epsilon_r/\delta_r^\alpha$ is bounded above and away from 0, where
$\epsilon_r=(n_1n_2\cdots n_r)^{-1}$ and $\delta_r=(m_1m_2\cdots
m_r)^{-1}$.

Next, we construct inductively a sequence $C_0,C_1,C_2,\cdots$ where
$C_r$ is a chain of squares of side $\epsilon_r$. We start by letting
$C_0$ be a single square of side 1. Next, supposing $C_r$ constructed as
a chain of squares of side $\epsilon_r$, we divide each square $Q$
of $C_r$ into a $n_{r+1}\times n_{r+1}$ grid of squares of side
$\epsilon_{r+1}$. Two of the sides of $Q$ abut other squares of $C_r$
and we now construct a chain of squares of side $\epsilon_{r+1}$
consisting of squares of this grid, joining
the middle edge squares of these two sides, containing no other edge
squares, and consisting of $m_{r+1}$ squares. 

The sequence $C_r$ converges to a curve $C$, which can be parametrised
by $t\rightarrow u(t),\ t\in[0,1]$, in such a way that $u(t)$ spends
time $\delta_r$ in each square of $C_r$. To see that $u(t)$ satisfies
the required inequality, suppose $s,t\in[0,1]$ and suppose
$\delta_r<|s-t|\leq\delta_{r-1}$. Then $u(s)$ and $u(t)$ belong to the
same or adjoining squares of $C_{r-1}$, so
$u(s)-u(t)|\leq3\epsilon_{r-1}$. On the other hand $u(s)$ and $u(t)$ are
not in the same square of $C_r$, and not in adjoining squares of
$C_{r+1}$, so $|u(s)-u(t)|\geq\epsilon_{r+1}$, which completes the
proof.\end{proof}\noindent
{\bf Example 4.} Nonexistence of solutions for $f\in C^{p-1}$ when
$2<p<3$.

The construction in Example 3 works for $2<p<3$, with the following
modifications. We use the same definitions for $x^1$ and $z$ .We need a
modified Lemma \ref{l7}, proved in a similar way, which assumes
$\frac13<\alpha<1$ and gives $u=(x^2,x^3,x^4)$ taking values in $\R^3$.
Using this and, for example, the version of Whitney's extension theorem in
Theorem 4 of Section VI.2 of \cite{st}, we get $z(t)=f(x^2(t),x^3(t),x^4(t))$
where $f\in C^{p-1}$ and now $Df(x^2(t),x^3(t),x^4(t))=0$. We then consider
the same system as in example 3 (with $x^4,y^4$ added in obvious fashion).
Then, because $Df(y^2,y^3,y^4)$ is always 0, whatever choice is made for
$A^{rj}$, the term involving $A^{rj}$ in (\ref{eq10}) always vanishes, and
any solution in the sense of Definition \ref{def2} will be a naive solution for $t>0$.
Then the same argument as before shows that no solution exists.

\section{Global existence and explosions.}\label{ge}

When $x(t)$ is defined on $[0,\infty)$ and $f$ is globally defined,
Theorems  \ref{th1}-\ref{th4} show that, under suitable conditions, equation
(\ref{eq1}) has either a solution for all positive $t$ or a solution such
that $|y(t)|$ goes to $\infty$ at some finite time (an explosion). In this
section we investigate what conditions will ensure that no explosion occurs,
so that a solution exists for all time.

For equations of the form (\ref{eq1}) where $x(t)$ has (locally) bounded
variation, it is not hard to show that if $D(R)$ is a positive increasing
function for $R\geq1$ with $\int_1^\infty D(R)^{-1}dR=\infty$ and if $f$ is
continuous on $\R^n$ and satisfies $|f(y)|\leq D(|y|)$ for all $y$ then
no solution can explode in finite time. The following theorem gives an
analogous result for the case when $x$ has finite $p$-variation for $p>1$.
In this case we require control of the growth of H\"{o}lder continuity
bounds of $f$ as well as $|f|$ itself.

We suppose that either (i) $1<p<\gamma<2$ or (ii) $2\leq
p<\gamma<3$ and let $\beta=\gamma-1$.

\begin{thm}\label{th9} (a) Suppose $D(R)$ and $A(R)$ are positive increasing
functions on $1\leq R<\infty$ with $D(R)\leq R^\beta A(R)$, such that
$|f(y)|\leq D(R)$ and that, in case (i) $|f(y')-f(y)|\leq
A(R)|y'-y|^\beta$, while in case (ii) $|Df(y')-Df(y)|\leq A(R)|y'-y|
^{\beta-1}$ for $|y,y'|\leq R$. Suppose $x(t)$ has finite $p$-variation
on each bounded interval, and in case (ii) $A(s,t)$ satisfies assumption
1 on each bounded interval. Then, provided
\[\int_1^\infty\left\{A(R)^{1-p}D(R)^{p-1-\beta p}\right\}^{1/\beta}dR
=\infty\]
no solution of (\ref{eq1}) can explode in finite time.\vspace{.2cm}\\
(b) Conversely, suppose $D(R)$ and $A(R)$ are positive increasing
functions on $1\leq R<\infty$ with $D(R)\leq R^\beta A(R)$, 
and suppose
\[\int_1^\infty\left\{A(R)^{1-p}D(R)^{p-1-\beta p}\right\}^{1/\beta}dR
<\infty\]
Then we construct $f$, $x(t)$, and in case (ii) $A(s,t)$, with the same
conditions as in (a), such that (1)
has a solution which explodes in finite time.
\end{thm}
{\bf Remark.} The condition $D(R)\leq R^\beta A(R)$ is natural, since
the second condition on $f$ in part (a) implies the existence of a constant $C$ such that
$|f(y)|\leq C+A(R)R^\beta$ for $|y|\leq R$.

\begin{proof} (a) This is essentially a case of keeping track
of the bounds in the arguments of Sections \ref{s2} and \ref{s3}. We start
with case (i).

Let $A_k=A(2^k)$ and $D_k=D(2^k)$. Then
$\sum2^k\left\{A_k^{1-p}D_k^{p-1-\beta p}\right\}^{1/\beta}=\infty$. Let
$k_0$ be the smallest nonnegative integer such that $2^k>|y(0)|$. Then
for $k=k_0,k_0+1,\cdots$ let $t_k$ be the first time that $|y(t)|=2^k$
(if for a given $k$ no such time exists, then there can be no explosion).

Then for $t_k\leq t\leq t_{k+1}$ we have $|y|\leq2^{k+1}$, in which
region $|f(y)|\leq D_{k+1}$ and $|f(y)-f(y')|\leq A_{k+1}|y-y'|^\beta$.

Now we apply the estimates of Lemma \ref{l1} (which, in view of Remark 1,
apply to any solution) on the interval $[t_k,t_{k+1}]$,
and note that in the proof of Lemma \ref{l1} we can take $B_1=D_{k+1}$ and
$B_2=2A_{k+1}D_{k+1}^\beta$. We can then take
$L=4A_{k+1}D_{k+1}^\beta(1-2^{1-\gamma/p})^{-1}$ and
$\delta=(D_{k+1}^{1-\beta}A_{k+1}^{-1}(1-2^{1-\gamma/p})/4)^{p/\beta}$.
Then on any time interval with $\Delta\omega\leq\delta$ we have $\Delta
y\leq c_1(D/A)^{1/\beta}$ where
$c_1=2(\frac14(1-2^{1-\gamma/p}))^{1/\beta}<\frac14$. Then, since
$|y(t_{k+1})-y(t_k)|\geq2^k$, the number of intervals of length $\delta$
that fit into $[\omega(t_k),\omega(t_{k+1})]$ is at least the integer
part of $c_1^{-1}2^k(A_{k+1}/D_{k+1})^{1/\beta}$ which, in view of the fact that
$(D_{k+1}/A_{k+1})^{1/\beta}\leq2^{k+1}$, is
$\geq c_1^{-1}2^{k-1}(A_{k+1}/D_{k+1})^{1/\beta}$.

Hence
\[\omega(t_{k+1})-\omega(t_k)\geq2^{k-1}c^{-1}(A_{k+1}/D_{k+1})^{1/\beta}
\delta={\rm const}\ 2^k\{A_{k+1}^{1-p}D_{k+1}^{p-1-\beta p}\}^{1/\beta}\]
so $\sum(\omega(t_{k+1})-\omega(t_k))=\infty$, so
$\omega(t_k)\rightarrow\infty$ as $k\rightarrow\infty$, which means there
is no explosion.

The arguments for case (ii) is similar, using now the estimates of Lemma
\ref{l2}. First note that n this case we have for $t_k\leq t\leq t_{k+1}$
that $|Df(y)-Df(y')|\leq A_{k+1}|y-y'|^{\beta-1}$ which together with
$|f|\leq D_{k+1}$ gives by interpolation (using the fact that $D_{k+1}\leq
2^{(k+1)\beta}A_{k+1}$) that $|Df(y)|\leq c(A_{k+1}D_{k+1}^{\beta-1})^
{1/\beta}$.

Now we apply the estimates of Lemma \ref{l2} on $[t_k,t_{k+1}]$ and note that
we can take $B_1=A_{k+1}$, $B_2=c_2D_{k+1}(A_{k+1}D_{k+1})^{1/\beta}$ and
$B_3=D_{k+1}$ provided $\delta<(D_{k+1}^{1-\beta}A_{k+1}^{-1})^{-1}$
and then bounding each term in (12) we find we can take
$B_4=c_3A_{k+1}D_{k+1}^\beta$. So apart from constants we get the same
bounds for $\delta$ and $L$ as in case (i) and the proof concludes in the
same way.\vspace{.2cm}\\
(b) We construct a system of the form (\ref{eq1}) with $d=2$ and $n=1$.

Let $\rho'=(\beta p+1-p)/\beta$, $\rho''=(p-1)/\beta$,
$\rho=\min(\rho',\rho'')$ and for $y\geq1$ let
$F(y)=A(y)^{-\rho''}D(y)^{-\rho'}$, so that the hypothesis gives
$\int_1^\infty F(y)dy<\infty$.

We need to `smooth' the functions $A$ and $D$. Choose $r>\rho^{-1}$
and define $\tilde{D}(y)=\inf_{u\geq1}u^rD(y/u)$,
$\tilde{A}(y)=\inf_{u\geq1}u^rA(y/u)$. Then $\tilde{D}(y)\leq D(y)$ and
$\tilde{A}(y)\leq A(y)$. We let $\tilde{F}(y)=\tilde{A}(y)^{-\rho''}
\tilde{D}(y)^{-\rho'}$. Then
$\tilde{F}(y)\leq\sup_{u\geq1}u^{-r\rho}F(y/u)$. Now if we extend $F$ to
$[0,\infty)$ by setting $F(y)=F(1)$ for $0\leq y<1$ then we have
\[u^{-r\rho}F(y/u)\leq r\rho F(y/u)\int_u^\infty v^{-r\rho-1}dv
\leq r\rho\int_u^\infty F(y/v)v^{-r\rho-1}dv\]and so $\tilde{F}(y)
\leq r\rho\int_1^\infty F(y/v)v^{-r\rho-1}dv$ for $y\geq1$.

Hence
\[\int_1^\infty\tilde{F}(y)dy\leq r\rho\int_1^\infty\int_1^\infty
F(y/v)v^{-r\rho-1}F(y/v)dvdy\leq r\rho\int_1^\infty v^{-r\rho}dv
\int_0^\infty F(y)dy<\infty\]

Next we fix a smooth non-negative function $\phi$ supported on the
interval [1,2] such that $\int\phi=1$, and define
$D^*(y)=2^{-r}\int\tilde{D}(yu)\phi(u)du$ for $y\geq1$, and $A^*$ similarly.
Then $2^{-r}\tilde{D}(y)\leq D^*(y)\leq\tilde{D}(y)$ with a similar
inequality for $A^*$. Then we set
$F^*(y)=A^*(y)^{-\rho''}D^*(y)^{-\rho'}$ and we have $\int_1^\infty
F^*(y)dy<\infty$.

For $y\geq1$ we write $\lambda(y)=\int_1^y(A*/D^*)^{1/\beta}$,
$\A(y)=(D^*(y)^{1-\beta}A^*(y)^{-1})^{1/\beta}$ and
\[f(y)=D^*(y)(-\sin\lambda(y),\cos\lambda(y))\in\R^2\]
Let $t_*=\int_1^\infty F^*$ and define $y(t)$ on $[0,t_*)$ by
$t=\int_1^y F^*$ and note that $y(t)\rightarrow\infty$ as
$t\rightarrow t_*$. Then let $x(t)=\A(y(t))(\cos\lambda(y(t)),
\sin\lambda(y(t)))\in\R^2$ for $0\leq t<t_*$, and let $x(t)=0$ for
$t\geq t_*$. Then the equation $dy(t)=f(y).dx(t)$ is satisfied, as a
classical ODE on $[0,t_*)$. If we can show that $x$ has finite
$p$-variation and $f$ satisfies the required $\beta$-H\"{o}lder bound
then in case (i) $y$ will satisfy (\ref{eq1}) in the sense of Definition \ref{def1} and
we will have the required example. In case (ii) we define
$A^{ij}(s,t)=-x^i(s)\{x^j(t)-x^j(s)\}$; one can then check that the term
involving $A^{rj}$ in (\ref{eq10}) vanishes, so that (\ref{eq10}) will hold on
any compact subinterval of $[0,t_*)$, and again we have the required example
provided $A^{ij}$ satisfies the $\frac p2$-variation condition.

As preparation for proving the required bounds we note that the
assumption $D(y)\leq y^\beta
A(y)$ implies $D^*(y)\leq y^\beta A^*(y)$ and so
$\lambda'(y)\geq y^{-1}$. Hence if $y_1<y$ and
$\lambda(y)-\lambda(y_1)\leq1$ then it follows that $y\leq ey_1$, and
hence $D^*(y)\leq e^rD^*(y_1)$ with a similar inequality for
$A^*$, so the relative variation of each of the functions
$D^*,A^*,\lambda',\alpha$ is bounded by a constant on
$[y_1,y]$. Also we have
\[D^*(y)-D^*(y_1)\leq C\-_1D^*(y_1)(y-y_1)/y_1\leq C_2D^*(y_1)(\lambda(y)-\lambda
(y_1))\]
with similar bounds for $A^*$ and $\alpha$.
 
We now consider $x(t)$, and show in fact that it satisfies a
$\frac1p$-H\"{o}lder condition. Let $t_1<t_2<t_*$, and we write $y_1$
for $y(t_1)$ etc. First we suppose that $\lambda_2-\lambda_1\leq1$. Then
the discussion in the preceding paragraph shows that $|\A_2-\A_1|\leq
C_3\A_1(\lambda_2-\lambda_1)$ and so
\[|x_1-x_2|\leq|\A_2-\A_1|+\A_1(\lambda_2-\lambda_1)\leq
C_4\A_1(\lambda_2-\lambda_1)\]
Also $\lambda_2-\lambda_1\leq C-5\A_1^{-p}(t_2-t_1)$ so $\A_1\leq
C_6(\frac{t_2-t_1}{\lambda_2-\lambda_1})^{1/p}$ and so
\[|x_2-x_1|\leq C_7(t_2-t_1)^{1/p}(\lambda_2-\lambda_1)^{1-1/p}
\leq(t_2-t_1)^{1/p}\]
proving the H\"{o}lder estimate
in this case. In the case $\lambda_2-\lambda_1>1$ we have $|x_1|,|x_2|
\leq C(t_2-t_1)^{1/p}$ giving the estimate in this case also, and the
case $t_2\geq t_*$ follows similarly. In case (ii) the $\frac
p2$-variation condition for $A^{ij}$ follows easily.

The treatment of $f(y)$ is similar. We consider case (i) first. When
$\lambda_2-\lambda_1\leq1$ 
we have $|f_2-f_1|\leq C_8D^*_1(\lambda_2-\lambda_1)$ and $\lambda_2-
\lambda_1\leq C_9(A^*_1/D^*_1)^{1/\beta}(y_2-y_1)$ so
$D^*_1\leq C_{10}A^*_1(\frac{y_2-y_1}{\lambda_2-\lambda_1})^\beta$ so
\[|f_2-f_1|\leq C_{11}A^*_1(y_2-y_1)^\beta\leq C_{11}A_1(y_2-y_1)^\beta\]
The case $\lambda_2-\lambda_1>1$ is treated in the same way as before.

For case (ii), first suppose $\lambda_2-\lambda_1\leq1$. We have the
derivative bounds $\lambda'=(A^*/D^*)^{1/\beta}$, $|\lambda''|\leq
C_{12}y^{-1}(A^*/D^*)^{1/\beta}$, $(D^*)'\leq C_{12}y^{-1}D^*$, $|(D^*)''|\leq
C_{12}y^{-2}D^*$ from which we deduce, remembering that $D^*(y)\leq y^\beta
A^*(y)$, that $|f''(y)|\leq C_{13}D^*(A^*/D^*)^{2/\beta}$. Hence we have
\[\begin{split}|f_2'-f_1'|&\leq C_{14}D^*_1(A^*_1/D^*_1)^{2/\beta}(y_2-y_1)=
C_{14}D^*_1(A^*_1/D^*_1)^{2/\beta}(y_2-y_1)^{2-\beta}(y_2-y_1)^{\beta-1}\\
&\leq C_{15}D^*_1(A^*_1/D^*_1)^{2/\beta}(A^*_1/D^*_1)^{1-2/\beta}(y_2-y_1)^{\beta-1}\\
&=C_{15}A^*_1(y_2-y_1)^{\beta-1}\leq C_{15}A_1(y_2-y_1)^\beta\end{split}\]
as required. The case $\lambda_2-\lambda_1>1$ is treated as before, using
the bound $|f'|\leq C_{13}D^*(A^*/D^*)^{1/\beta}$.
\end{proof}

\section{Convergence of Euler approximations}\label{ce}

In the situation of Section \ref{s2} ($1\leq p<\gamma\leq2$), Theorem
\ref{th1} establishes the convergence of Euler approximations (\ref{eq3}) to
the solution as the mesh size of the partition tends to 0. In fact the proof
gives a bound for the rate of convergence: from Remark 1 (and the fact that
$f\in C^\gamma$ implies $f\in C^1$) we see that the
solution satisfies (\ref{eq1}) with $\theta(\delta)=C\delta^{2/p}$; then the
last paragraph of the proof of Theorem \ref{th1} gives
\[|y_K-y(t)|\leq{\rm const}\sum_{k=0}^{K-1}\omega_{k,k+1}^{2/p},\]
where $\{y_k\}$ is given by (\ref{eq3}) for a partition such that $t_K=t$.

In the situation of Section \ref{s3} ($2\leq p<\gamma\leq3$) similar reasoning
leads to a bound
\[|y_K-y(t)|\leq{\rm const}\sum_{k=0}^{K-1}\omega_{k,k+1}^{3/p}\]
where now $\{y_k\}$ is given by the scheme (\ref{eq11}).

Neither of the above results covers the known fact that Euler
approximations of the form (\ref{eq3}), containing no $A^{rj}$ term,
converge almost surely to solutions of It\^{o} equations
driven by Brownian motion. In this section we obtain a convergence
result in the setting of Section \ref{s3}, but using the Euler approximation
(\ref{eq3}) rather than (\ref{eq11}); for this to work we need to impose an
additional condition on the driving path, which can be thought of as a
`pathwise' version of the `independent increments' property of Brownian
motion. 

One form of the convergence result for It\^{o} equations states that, if
$x(t)$ is a $d$-dimensional Brownian motion and $f$ satisfies a global
Lipschitz condition, and $T>0$ is fixed, then with probability 1, for any
$\epsilon>0$ there is a constant $C$ such that if $0<t\leq T$ then
\Z\label{eq20}|y(t)-y_K|\leq Ct^{1-\epsilon}K^{-\frac12+\epsilon}\z
where $y(t)$ is the solution of (\ref{eq1}), interpreted as an It\^{o} equation,
and $\{y_k\}$ is given by the Euler scheme (\ref{eq2}) with $t_k=kt/K$. There is
also a convergence result for non-uniform step sizes, provided the mesh
points are stopping times. But convergence can fail if no restriction is
imposed on the partition, as shown in \cite{gl}.

For simplicity we use uniform step sizes, and then it is convenient to
assume a H\"{o}lder condition of order $\alpha=\frac1p$ on the driving path
rather than a $p$-variation condition.

We suppose $\frac13<\alpha<\frac12$ and $1-\alpha<\beta<2\alpha$. We
assume that the driving path $x\in C^\alpha[0,T]$ and that there is a
constant $B$ such that $A^{ij}(s,t)$ satisfies
\Z\label{eq21}\left|\sum_{l=k}^{m-1}A^{ij}(lh,(l+1)h)\right|\leq B(m-k)^\beta
h^{2\alpha}\z
whenever $0<k<m$ are integers and $h>0$ such that $mh\leq T$. Under
these hypotheses we have:

\begin{thm}\label{th10} Suppose $f\in C^\gamma$ where $\gamma>\alpha^{-1}$.
Let $0<t\leq T$, let $K$ be a positive integer, and let $z_k$ be defined by
the Euler recurrence relation
\[z_{k+1}^i=z_k^i+f^i_j(z_k)(x^j(t_{k+1})-x^j(t_k))\]
where $t_k=kt/K$ and $z_0=y_0$.

Let also $y$ be the solution of (\ref{eq1}) in the sense of Definition \ref{def2}.

Then $|z_k-y(t)|\leq Ct^{2\alpha}K^{\beta-2\alpha}$, where $C$ is a
constant independent of $K$ and $t$.
\end{thm}
\begin{proof} For $0\leq l<l\leq K$ we define $T_{kl}(z)$ to be $y_l$
where $\{y_m\}$ satisfies (\ref{eq11}) with $y_k=z$. From the estimates in
Section \ref{s3} we have
\Z\label{eq22}|T_{kl}(z)-T_{kl}(z')-(z-z')|\leq C_1\{(l-k)h\}^\alpha|z-z'|\z
where $h=t/K$. We also write
$R_k^{ij}=\sum_{l=0}^{k-1}A^{ij}(kh,(k+1)h))$ and
$g^i_r(y)=f^h_r(y)\partial_hf^i_j(y)$. Then we define (suppressing
indices for notational simplicity)
\[u_{kl}=z_l-T_{kl}(z_k)+g(z_k)(R_l-R_k)\]
for $0<k<l\leq K$.

Now if $0\leq k<l<m\leq K$ then
\[\begin{split}z_m&=T_{lm}(z_l)-g(z_l)(R_m-R_l)+u_{lm}\\
&=T_{lm}(T_{kl}(z_k)-g(z_k)(R_l-R_k)+u_{kl})-g(z_l)(R_m-R_l)+u_{lm}\\
&=T_{km}(z_k)+v_{klm}-g(z_k)(R_l-R_k)-g(z_l)(R_m-R_l)+u_{kl}+u_{lm}
\end{split}\]where
\[v_{klm}=T_{lm}(T_{kl}(z_k)-g(z_k)(R_l-R_k)+u_{kl})-T_{lm}(T_{kl}(z_k))
+g(z_k)(R_l-R_k)-u_{kl}\]and we have used the fact
that $T_{lm}(T_{kl}(z))=T_{km}(z)$. Hence
\Z\label{eq23}u_{km}=v_{klm}+w_{klm}+u_{kl}+u_{lm}\z
where $w_{klm}=(g(z_k)-g(z_l))(R_m-R_l)$. By (\ref{eq21}) we have the bound $|R_m-R_l|\leq
B(m-l)^\beta h^{2\alpha}$ and
\[|z_k-z_l|=|u_{kl}+(T_{kl}(z_l)-z_k)-g(z_k)(R_l-R_k)|\leq
C_2\{((l-k)h)^\alpha+|u_{kl}|\}\]so
\[|w_{klm}|\leq C_3\{((l-k)h)^\alpha\}+|u_{kl}|\}(m-l)^\beta h^{2\alpha}\]
And from (\ref{eq22}) we obtain
$|v_{klm}|\leq C_4\{(m-l)h\}^\alpha\{|u_{kl}|+(l-k)^\beta h^{2\alpha}\}$.
Putting these bounds into (\ref{eq23}) gives
\[|u_{km}|\leq|u_{kl}|\{1+C_5((m-k)h)^\alpha\}+|u_{lm}|+C_5(m-k)^{\alpha+\beta}
h^{3\alpha}\]

Also $u_{k,k+1}=0$. It then follows by an inductive argument similar to
that used in the proof of Lemma \ref{l2} that $|u_{km}|\leq
C(m-k)^{\alpha+\beta}h^{3\alpha}$ for $0\leq k<m\leq K$. We apply this
to $u_{0K}$ and use the fact that, by (\ref{eq9}) and Remark 3, $|T_{0K}(z_0)-
y(t)|\leq CKh^{2\alpha}$ together with the bound $|R_K-R_0|\leq BK^\beta
h^{2\alpha}$ to get the required bound for $z_K-y(t)$.
\end{proof}

The condition (\ref{eq21}) holds for Brownian motion (with the It\^{o}
interpretation of $A^{rj}$) for all $\alpha<\frac12$ and $\beta>\frac12$ with
probability 1, and so the bound (\ref{eq20}) for the error of the Euler
approximation follows from Theorem \ref{th9}. Indeed Theorem \ref{th9}
implies that, for almost all Brownian paths $x(t)$, for any $f\in C^\gamma$
where $\gamma>2$ and $T>0$, there is a constant $C$ such that the bound
(\ref{eq20}) holds for uniform-step Euler approximations to the solution of
(\ref{eq1}) on $[0,T]$. The construction in Theorem \ref{th6} can be modified
to show that this last statement can fail for $f\in C^\gamma$ if $\gamma<2$.
Indeed, for almost all 6-dimensional Brownian paths one can construct $f$ in
$C^\gamma$ for all $\gamma<2$ such that the Euler approximations to (1) fail
to converge.

We also note that Theorem \ref{th10} implies that the solution can be obtained
from the path $x(t)$ alone, since the $A(s,t)$ do not appear in the approximation.
This indicates that when (\ref{eq21}) holds the $A(s,t)$ are determined by the path
$x(t)$. And indeed it is not hard to deduce from (\ref{eq21}) that $A^{ij}(s,t)$ is
the limit as $N\rightarrow\infty$ of $\sum_{k=0}^Nx^i(t_k)\{x^j(t_{k+1})-x^j(t_k)\}$
where $t_k=s+k(t-s)/N$. Then (\ref{eq21}) is effectively a condition on the path $x(t)$.

\section*{Acknowledgements}

The author is grateful to Terry Lyons for many valuable discussions and to Peter Friz for
drawing his attention to \cite{fr}.

\bibliographystyle{amsplain}

\end{document}